\documentclass[12pt]{amsart}
\usepackage{geometry} % see geometry.pdf on how to lay out the page. There's lots.

\usepackage{amsmath}
 \usepackage{amssymb}
\usepackage{amsthm}
\usepackage{tikz}
\usepgflibrary{arrows}
\usepackage{graphicx}
\usepackage{comment}
\usepackage{todonotes}
\usepackage{color,hyperref,soul}
\definecolor{darkblue}{rgb}{0.0,0.0,0.3}
\hypersetup{colorlinks,breaklinks,
            linkcolor=darkblue,urlcolor=darkblue,
            anchorcolor=darkblue,citecolor=darkblue}
            
\newcommand{\blue}[1]{\textcolor{blue}{#1}}
\newcommand{\red}[1]{\textcolor{red}{#1}}
\newcommand{\green}[1]{\textcolor{green!60!black}{#1}}

\newtheorem{theorem}{Theorem}[section]
\newtheorem{lemma}[theorem]{Lemma}
\newtheorem{proposition}[theorem]{Proposition}
\newtheorem{corollary}[theorem]{Corollary}

\theoremstyle{definition}
\newtheorem{example}[theorem]{Example}

\newtheorem{definition}[theorem]{Definition}
\newtheorem{question}[theorem]{Question}
\newtheorem{remark}[theorem]{Remark}

\newcommand{\QQ}{\mathbb{Q}}

\newcommand{\B}{\mathcal{B}}

\renewcommand{\O}{\mathcal{O}}

\renewcommand{\P}{\mathbb{P}}

\newcommand{\x}{{\mathbf x}}

\newcommand{\Sym}{\mathrm{Sym}}

\newcommand{\asc}{\mathop{\mathrm asc}}
\newcommand{\ps}{\mathop{\mathrm ps^1}}
\newcommand{\sign}{\mathop{\mathrm sign}}
\newcommand{\Vol}{\mathop{\mathrm Vol}}
\newcommand{\Park}{\mathrm{Park}}
\newcommand{\PSC}{\mathrm{OPP}}

\newcommand{\Hasse}{\mathop{Hasse}}

\specialcomment{moredetails}{\begingroup\color{red}}{\endgroup}
\newcommand{\caro}[1]{\todo[color=orange!30]{#1 \\ \hfill --- C.}}
\newcommand{\nantel}[1]{\todo[color=red!25]{#1 ---N.}}
\newcommand{\john}[1]{\todo[color=green!30]{#1 \\ \hfill --- J.}}

 \author{Carolina Benedetti}\address[Benedetti]
{Departamento de Matem\'aticas\\ Universidad de los andes\\ Bogot\'a\\ COLOMBIA}
\email{c.benedetti@uniandes.edu.co}
\urladdr{https://sites.google.com/site/carobenedettimath/}

\author{Nantel Bergeron}\address[Bergeron]
{Department of Mathematics and Statistics\\ York  University\\ To\-ron\-to, Ontario M3J 1P3\\ CANADA}
\email{bergeron@mathstat.yorku.ca}
\urladdr{http://www.math.yorku.ca/bergeron}
\thanks{With partial support of Bergeron's York University Research Chair and NSERC}

\author{John Machacek}\address[Machacek]
{Department of Mathematics and Statistics\\ York  University\\ To\-ron\-to, Ontario M3J 1P3\\ CANADA}
\email{machacek@yorku.ca}
\urladdr{https://www.yorku.ca/machacek/home.html}

\title[Hypergraphic polytopes]{Hypergraphic polytopes: combinatorial properties and antipode}
\date{} % delete this line to display the current date

\begin{document}

\maketitle

\begin{center}
{\em To the memory of Jeff Remmel}
\end{center}

\begin{abstract}
\smallskip
In an earlier paper, the first two authors defined orientations on hypergraphs. Using this definition we provide an explicit bijection between acyclic orientations in hypergraphs and faces of hypergraphic polytopes. This allows us to obtain a geometric interpretation of the coefficients of the antipode map in a Hopf algebra of hypergraphs. This interpretation differs from similar ones for a different Hopf structure on hypergraphs provided recently by Aguiar and Ardila. Furthermore, making use of the tools and definitions developed here regarding orientations of hypergraphs we provide a characterization of hypergraphs giving rise to simple hypergraphic polytopes in terms of acyclic orientations of the hypergraph. In particular, we recover this fact for the nestohedra and the hyper-permutahedra, and prove it for generalized Pitman-Stanley polytopes as defined here. 
\end{abstract}

\section{Introduction}

Given a collection of combinatorial objects, one often wants to study how these objects can be broken into simpler pieces and how they can be reassembled.
Joni and Rota observed that Hopf algebras provide a natural framework to do this~\cite{JoniRota}.
Here the coalgebra structure records the splitting, and the algebra structure records the assembly. The advantage of adding such structure to a given combinatorial family is that the coalgebra map allows to decompose into smaller pieces an object of the family. These pieces can be put back together somehow via the algebra structure. For instance, if one aims to color vertices of a graph in a way that neighbouring vertices have different colors, one may think of breaking the graph into pieces in such a way that each resulting piece is a subgraph with no edges, then color each piece and put them all back together to obtain a coloring of the original graph.

The algebraic and coalgebraic structure in a Hopf algebra allow to define another important piece of a Hopf algebra, namely, its antipode.
Given any graded connected Hopf algebra, the antipode is given by Takeuchi's formula~\cite{Tak}.
However, this formula can be rather complicated and it often contains many cancellations. In view of this a common problem surrounding such a Hopf algebra is: what is a cancellation free formula for its antipode? We will refer to this problem as the antipode problem. Part of the interest in finding a solution to the antipode problem is that its formula encodes information about the underlying combinatorial object. 

A solution for the antipode problem in the Hopf algebra of graphs\footnote{\emph{The} Hopf algebra of graphs (simplical complexes, hypergraphs) should really be \emph{a} Hopf algebra of graphs (simplical complexes, hypergraphs) since there are multiple Hopf algebra structures which can be defined. In this paper we will be explicit about the Hopf algebra construction that we will use.} was first found by Humpert and Martin~\cite{HM}.
Using sign reversing involutions Benedetti and Sagan, as well as Bergeron and Ceballos, were able to give solutions for the same problem for various Hopf algebras including the graph Hopf algebra~\cite{BS,BC}. In this case, the antipode formula encodes acyclic orientations of graphs.

The technique of sign reversing involutions has been used to solve the antipode problem for the Hopf algebra of simplicial complexes~\cite{BenedettiHallamMachacek}.
The first two authors have further generalized this way of obtaining optimal formulas for antipode maps and provided a formula for the antipode in the Hopf algebra of hypergraphs in~\cite{BB16}.
It is also shown how the understanding of a Hopf algebra structure on hypergraphs allows one to understand the structure of a larger class of Hopf algebras.

The antipode formula in the Hopf algebra of hypergraphs obtained in~\cite{BB16}  is much simpler than Takeuchi's formula, but it is not cancellation free. Thus it does not solve the antipode problem. However, one of the main results in this paper addresses this issue in a geometric fashion. This paper is organized as follows.

In Section~\ref{S:GeometricAntipode} we will give a geometric interpretation of the coefficients of the antipode of a hypergraph in terms of a polytope called the hypergraphic polytope.
This geometric interpretation will explain the cancellation in the antipode formula by showing that the coefficients in the antipode map are Euler characteristics.

Using our notion of orientations on hypergraphs defined in Section~\ref{S:GeometricAntipode}, we then turn our attention to the hypergraphic polytope itself and derive some geometric results in Section~\ref{S:polytope}.  More specifically we characterize hypergraphic polytopes that are simple by means of acyclic orientations. This particular result is illustrated with some specific families of hypergraphic polytopes: the nestohedra and the hyperpermutahedra. Moreover, we define and study the family of \emph{generalized Pitman-Stanley polytopes} which, as their name indicates, contain as a particular case the Pitman-Stanley polytope.

%We conclude in Section~\ref{S:csf} with a brief study of the chromatic symmetric function of hypergraphs.
%Aguiar, Bergeron, and Sottile enriched the study of Hopf algebras in combinatorics by giving the modern definition of combinatorial Hopf algebra (CHA)~\cite{ABS}.
%A CHA is a pair $(H, \zeta)$ where $H$ is a graded connected Hopf algebra over $\QQ$  and $\zeta: H \to \QQ$ is an algebra homomorphism.
%Given a Hopf algebra whose basis is indexed by a collection $\mathcal{C}$, the map $\zeta$ allows one to associate a quasisymmetric function (that depends on $\zeta$) to each $c \in \mathcal{C}$.This function encodes information about how $c$ can be split and put back together. It also serves as an analogue of chromatic symmetric function. We investigate this with our Hopf algebra of hypergraphs.

%%%%%%%%%%%%%%%%%%%%%%%%%%%
%%%%%%%%%%%%%%%%%%%%%%%%%%%
%%%%%%%%%%%%%%%%%%%%%%%%%%%
\section{Geometric Antipode for Hypergraphs} \label{S:GeometricAntipode}

As described in the introduction, a recurrent and often difficult problem in Hopf algebras is to find a cancellation free formula for the antipode of a Hopf algebra.
The first two authors showed in~\cite{BB16},  that the Hopf algebra of hypergraphs encode the antipode problem for a large family of Hopf algebras 
and they give a description of the antipode for hypergraphs in term of acyclic orientations on them.
This new formula, interesting on its own, still contains many cancellations. 
Here we show that the \emph{hypergraphic polytope} $P_G$ associated to a hypergraph $G$ encodes the coefficients in the antipode $S(G)$.
This differs from the case  of graphical zonotopes in~\cite{Aguiar-Ardila} which considers a different Hopf structure on hypergraphs.

%%%%%%%%%%%%%%%%%%%%%%%%%%%
\subsection{Hypergraphs and orientations}\label{ss:hyper_orient}
Let  $2^{V}$ denote the collection of subsets of a finite set $V$. Let 
      $${\bf HG}[V]=\big\{G \subseteq 2^V\,\mid  U\in G\text{ implies }|U|\ge 2\big\}$$
An element $G\in{\bf HG}[V]$ is a \emph{hypergraph on $V$}.
We pause to remark that with some conventions, elements of ${\bf HG}[V]$ are \emph{simple} hypergraphs since repeated subsets of $V$ are not allowed.
However, we will omit the adjective simple as all hypergraphs we consider will be of this type.

\begin{example}\label{ex:hypergraph}
Consider $V= \{a,b,c,d,e,f\}$ and let 
$$G=\big\{\{b,c\},\red{\{a,b,e\}},\blue{\{a,d,e,f\}},\green{\{b,c,e\}},\{f,c\}\big\}\in{\bf HG}[V].$$ 
We graphically  represent $G$ as follows:
$$
G=
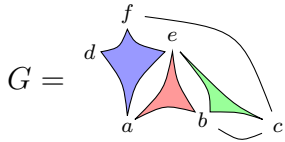
\begin{tikzpicture}[scale=1,baseline=.5cm]
	\node (a) at (0,0) {$\scriptstyle a$};
	\node (b) at (1,.1) {$\scriptstyle b$};
	\node (c) at (2,0) {$\scriptstyle c$};
	\node (d) at (-.5,1) {$\scriptstyle d$};
	\node (e) at (.6,1.2) {$\scriptstyle e$};
	\node (f) at (0,1.5) {$\scriptstyle f$};
	\draw [fill=red!40] (.1,.1) .. controls (.5,.3) .. (.9,.2) .. controls (.65,.6) .. (.6,1) .. controls (.5,.5) .. (.1,.1) ; 
	\draw [fill=blue!40] (0,.15) .. controls (-.05,.6) .. (-.35,1) .. controls (-.1,1.05) .. (0,1.3)  .. controls (.1,1.1) .. (.5,1) .. controls (.1,.6) .. (0,.15) ; 
	\draw [fill=green!40] (1.8,.1) .. controls (1.5,.2) .. (1.1,.2) .. controls (1.1,.5) .. (.7,1) .. controls (1.4,.3) .. (1.8,.1) ; 
	\draw (b) .. controls (1.5,-.2) ..   (c); 
	\draw (f) .. controls (1.5,1.3) ..   (c); 
\end{tikzpicture} 
$$
\end{example}

\begin{remark} With our notation, it is important to specify the vertex set $V$ on which the hypergraph $G$ is constructed. For example $G=\emptyset$ is not the same 
hypergraph when constructed on $V=\emptyset$ or $V=\{1,2,3\}$:
$$
\begin{array}{ccc}
\blue{\emptyset} \qquad & 
\begin{tikzpicture}[scale=.7,baseline=.2cm]
	\node (a) at (.3,.6) {$\scriptstyle \bullet \blue{2}$};
	\node (b) at (1,1) {$\scriptstyle  \bullet \blue{1}$};
	\node (d) at (-.5,1) {$\scriptstyle  \bullet \blue{3}$};
\end{tikzpicture} \qquad& 
\begin{tikzpicture}[scale=.7,baseline=.2cm]
	\node (a) at (.3,.6) {$\scriptstyle \bullet \blue{2}$};
	\node (b) at (1,1) {$\scriptstyle  \bullet \blue{1}$};
	\node (d) at (-.5,1) {$\scriptstyle  \bullet \blue{4}$};
	\node (d) at (-1.2,.3) {$\scriptstyle  \bullet \blue{5}$};
	\node (d) at (1.7,.8) {$\scriptstyle  \bullet \blue{3}$};
\end{tikzpicture} \\
\emptyset\in {\bf HG}[\emptyset]&\qquad\emptyset\in {\bf HG}[\{1,2,3\}]\qquad&\emptyset\in {\bf HG}[\{1,2,3,4,5\}]\\
\end{array}
$$
\end{remark}

In~\cite{BB16}, we introduced a notion of orientation for hypergraphs that is related to our antipode formula.
We recall here the basic definitions.

\begin{definition}[Orientation] Given a hypergraph $G$ an \emph{orientation} $(\mathfrak{a},\mathfrak{b})$ of a hyperedge $U\in G$ is an ordered set partition $(\mathfrak{a},\mathfrak{b})$ of $U$. %such that $U=\mathfrak{a}\cup\mathfrak{b}$ and $\mathfrak{a}\cap\mathfrak{b}=\emptyset$.
We can think of the orientation $(\mathfrak{a},\mathfrak{b})$ as current or flow on $U$ from a single vertex $\mathfrak{a}$ to the vertices in $\mathfrak{b}$ in which case we say that $\mathfrak{a}$ is the \emph{head} of the orientation $\mathfrak{a}\rightarrow\mathfrak{b}$ of $U$.
It what follows we will want to think of the vertices in $\mathfrak{a}$ as being contracted to a single point while the vertices in $\mathfrak{b}$ remain as distinct points.
If $|U|=n$, then there are a total of $2^{n}-2$ possible orientations. An  \emph{orientation of $G$} is an orientation of all its hyperedges. Given an orientation ${\mathcal O}$  on $G$, we say that $(\mathfrak{a},\mathfrak{b})\in{\mathcal O}$ if $(\mathfrak{a},\mathfrak{b})$ is the orientation of a hyperedge $U$ in $G$.
\end{definition}

\begin{example}\label{hypergraph}
With $G=\big\{\{b,c\},\red{\{a,b,e\}},\blue{\{a,d,e,f\}},\green{\{b,c,e\}},\{f,c\}\big\}$, we can orient the edge $U=\{a,b,e\}$ in $2^3-2=6$ different ways;  three with a head of size 1: $(\{a\},\{b,e\})$, $(\{b\},\{a,e\})$, $(\{e\},\{a,b\})$, and three with a head of size 2: $(\{b,e\},\{a\})$, $(\{a,e\},\{b\})$, $(\{a,b\},\{e\})$. We represent this graphically as follows:
$$
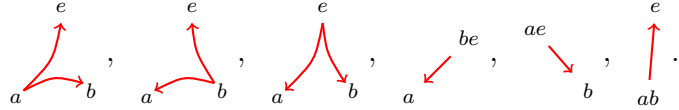
\begin{tikzpicture}[scale=1,baseline=.5cm]
	\node (a) at (0,0) {$\scriptstyle a$};
	\node (b) at (1,.1) {$\scriptstyle b$};
	\node (e) at (.6,1.2) {$\scriptstyle e$};
	\draw [thick,color=red,->] (.1,.1) .. controls (.5,.3) .. (.9,.2) ; 
	\draw [thick,color=red,<-] (.6,1) .. controls (.5,.5) .. (.1,.1) ; 
\end{tikzpicture} ,\ 
\begin{tikzpicture}[scale=1,baseline=.5cm]
	\node (a) at (0,0) {$\scriptstyle a$};
	\node (b) at (1,.1) {$\scriptstyle b$};
	\node (e) at (.6,1.2) {$\scriptstyle e$};
	\draw [thick,color=red,->] (.9,.2) .. controls (.65,.6) .. (.6,1); 
	\draw [thick,color=red,<-] (.1,.1) .. controls (.5,.3) .. (.9,.2) ; 
\end{tikzpicture} ,\ 
\begin{tikzpicture}[scale=1,baseline=.5cm]
	\node (a) at (0,0) {$\scriptstyle a$};
	\node (b) at (1,.1) {$\scriptstyle b$};
	\node (e) at (.6,1.2) {$\scriptstyle e$};
	\draw [thick,color=red,->] (.6,1) .. controls (.5,.5) .. (.1,.1); 
	\draw [thick,color=red,<-] (.9,.2) .. controls (.65,.6) .. (.6,1) ; 
\end{tikzpicture} ,\ 
\begin{tikzpicture}[scale=1,baseline=.5cm]
	\node (a) at (0,0) {$\scriptstyle a$};
	\node (b) at (.8,.8) {$\scriptstyle be$};
	\draw [thick,color=red,->] (b) -- (a); 
\end{tikzpicture} ,\ 
\begin{tikzpicture}[scale=1,baseline=.5cm]
	\node (a) at (1,.1) {$\scriptstyle b$};
	\node (b) at (.3,.9) {$\scriptstyle ae$};
	\draw [thick,color=red,->] (b) -- (a); 
\end{tikzpicture} ,\ 
\begin{tikzpicture}[scale=1,baseline=.5cm]
	\node (a) at  (.6,1.2) {$\scriptstyle e$};
	\node (b) at (.5,0) {$\scriptstyle ab$};
	\draw [thick,color=red,->] (b) -- (a); 
\end{tikzpicture}.$$
To orient $G$, we have to make a choice of orientation for each hyperedge. For example we can choose 
 ${\mathcal O}=\big\{(\{b\},\{c\}),\red{(\{a\},\{b,e\})},\blue{(\{a,e\},\{d,f\})},\green{(\{b,c\},\{e\})},(\{f\},\{c\})\big\}$ and we represent this
 as 
 $${ G/\mathcal O}=
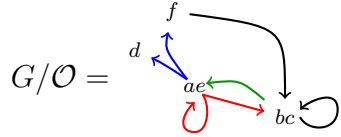
\begin{tikzpicture}[scale=1,baseline=.5cm]
	\node (a) at (.3,.5) {$\scriptstyle ae$};
	\node (b) at (1.5,.1) {$\scriptstyle bc$};
	\node (d) at (-.5,1) {$\scriptstyle d$};
	\node (f) at (0,1.5) {$\scriptstyle f$};
	\draw [thick,color=red,->]  (.4,.4).. controls (.5,.3) and (.5,-.1)..(.3,-.1) .. controls (.1,-.1) and (.1,.1).. (.3,.3); 
	\draw [thick,color=red,->] (.4,.4) -- (b); 
	\draw [thick,color=green!60!black,->] (b).. controls (.9,.6) .. (.45,.55); 
	\draw [thick,color=blue,->] (.2,.6).. controls (-.1,.8) .. (d); 
	\draw [thick,color=blue,->] (.2,.6).. controls (-0.1,1) .. (f); 
	\draw  [thick,->]  (f) .. controls (1.5,1.3) ..   (b); 
	\draw  [thick,->]  (1.7,.1) .. controls (1.8,-.1) and (2.2,-.1) ..(2.2,.2) ..  controls (2.2,.5) and (1.8,.3).. (1.7,.2); 
\end{tikzpicture} 
$$
Notice here, as we have previously stated, for an orientation $(\mathfrak{a}, \mathfrak{b})$ of a hyperedge we picture the vertices in $\mathfrak{a}$ as contracted to a single vertex.
A directed edge is then placed between this single vertex and each vertex in $\mathfrak{b}$.
\end{example}

In general, given a hypergraph $G$ on the vertex set $V$ and an orientation $\mathcal O$ of $G$, we construct an oriented (not necessarily simple) graph $G/{\mathcal O}$ as follows. 
We let $V/{\mathcal O}$ be the set partition of $V$ defined by the transitive closure of the relation $a\sim  a'$ if $a,a'\in \mathfrak{a}$ for some head $\mathfrak{a}$ of $\mathcal O$.
For each oriented hyperedge $(\mathfrak{a},\mathfrak{b})$ of $\mathcal O$, we have $|{\mathfrak b}|$ oriented edges $([{\mathfrak a}],[b])$ in $G/{\mathcal O}$ 
where $[{\mathfrak a}],[b]\in V/{\mathcal O}$ are equivalence classes and $b\in \mathfrak{b}$. 

\begin{definition}[Acyclic orientation] An orientation $\mathcal O$ of $G$ is \emph{acyclic} if the oriented graph $G/{\mathcal O}$ has no cycles.
\end{definition}

\begin{example}\label{ex:124_234}
 Let $G=\big\{\blue{\{1,2,4\}},\red{\{2,3,4\}}\big\}$ be a hypergraph on $V=\{1,2,3,4\}$. As we can see the orientations ${\mathcal O}=\big\{\blue{ (\{4\},\{1,2\})},\red{(\{2,4\},\{3\})}\big\}$ and ${\mathcal O}'=\big\{ \blue{(\{4\},\{1,2\})},$ $\red{(\{2,3\},\{4\})}\big\}$ are not acyclic, but  ${\mathcal O}''=\big\{\blue{ (\{4\},\{1,2\})},\red{(\{4\},\{2,3\})}\big\}$ is acyclic:
$$
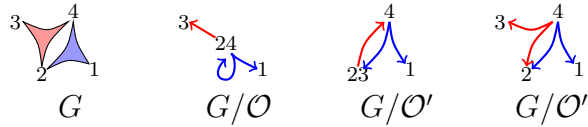
\begin{array}{cccc}
\begin{tikzpicture}[scale=.7,baseline=.2cm]
	\node (a) at (0,0) {$\scriptstyle 2$};
	\node (b) at (1,.1) {$\scriptstyle 1$};
	\node (d) at (-.5,1) {$\scriptstyle 3$};
	\node (e) at (.6,1.2) {$\scriptstyle 4$};
	\draw [fill=blue!40] (.1,.1) .. controls (.5,.3) .. (.9,.2) .. controls (.65,.6) .. (.6,1) .. controls (.5,.5) .. (.1,.1) ; 
	\draw [fill=red!40] (0,.15) .. controls (-.05,.6) .. (-.35,1) .. controls (.1,.8) .. (.5,1) .. controls (.1,.6) .. (0,.15) ; 
\end{tikzpicture}\quad & 
\begin{tikzpicture}[scale=.7,baseline=.2cm]
	\node (a) at (.3,.6) {$\scriptstyle 24$};
	\node (b) at (1,.1) {$\scriptstyle 1$};
	\node (d) at (-.5,1) {$\scriptstyle 3$};
	\draw [thick,color=blue,->]  (.4,.4).. controls (.5,.3) and (.5,-.1)..(.3,-.1) .. controls (.1,-.1) and (.1,.1).. (.3,.3); 
	\draw [thick,color=blue,->]  (.4,.4).. controls (.5,.3) ..(.9,.1); 
	\draw [thick,color=red,->]  (.1,.7)--(-.4,1); 
\end{tikzpicture} \quad& 
\begin{tikzpicture}[scale=.7,baseline=.2cm]
	\node (a) at (0,0) {$\scriptstyle 23$};
	\node (b) at (1,.1) {$\scriptstyle 1$};
	\node (e) at (.6,1.2) {$\scriptstyle 4$};
	\draw  [thick,color=blue,->]  (.6,1) .. controls (.5,.5) .. (.1,.1) ; 
	\draw  [thick,color=blue,<-]  (.9,.2) .. controls (.65,.6) .. (.6,1) ; 
	\draw [thick,color=red,<-]  (.5,1) .. controls (.1,.6) .. (0,.15) ; 
\end{tikzpicture} \quad& 
 \begin{tikzpicture}[scale=.7,baseline=.2cm]
	\node (a) at (0,0) {$\scriptstyle 2$};
	\node (b) at (1,.1) {$\scriptstyle 1$};
	\node (d) at (-.5,1) {$\scriptstyle 3$};
	\node (e) at (.6,1.2) {$\scriptstyle 4$};
	\draw  [thick,color=blue,->]  (.6,1) .. controls (.5,.5) .. (.1,.1) ; 
	\draw  [thick,color=blue,<-]  (.9,.2) .. controls (.65,.6) .. (.6,1) ; 
	\draw [thick,color=red,->]  (.5,1) .. controls (.1,.6) .. (0,.15) ; 
	\draw [thick,color=red,<-]  (-.35,1) .. controls (.1,.8) .. (.5,1) ; 
\end{tikzpicture} \quad\\
G&G/{\mathcal O}&G/{\mathcal O}'&G/{\mathcal O}''\\
\end{array}
$$
Out of the possible 36 orientations of $G$ only 20 are acyclic:
$$  \begin{array}{c}
 \scriptstyle
\{ (\{4\},\{1,2\}),(\{4\},\{2,3\})\};\quad \{ (\{4\},\{1,2\}),(\{3\},\{2,4\})\};\quad \{ (\{4\},\{1,2\}),(\{3,4\},\{2\})\};\quad \{ (\{2\},\{1,4\}),(\{3\},\{2,4\})\};\\
 \scriptstyle
 \{ (\{2\},\{1,4\}),(\{2\},\{3,4\})\};\quad \{ (\{2\},\{1,4\}),(\{2,3\},\{4\})\};\quad \{ (\{1\},\{2,4\}),(\{4\},\{2,3\})\};\quad\{ (\{1\},\{2,4\}),(\{3\},\{2,4\})\};\\
\scriptstyle
\{ (\{1\},\{2,4\}),(\{2\},\{3,4\})\};\quad\{ (\{1\},\{2,4\}),(\{2,3\},\{4\})\};\quad \{ (\{1\},\{2,4\}),(\{2,4\},\{3\})\};\quad\{ (\{1\},\{2,4\}),(\{3,4\},\{2\})\};\\
\scriptstyle
\{ (\{1,2\},\{4\}),(\{3\},\{2,4\})\};\quad\{ (\{1,2\},\{4\}),(\{2\},\{3,4\})\};\quad \{ (\{1,2\},\{4\}),(\{2,3\},\{4\})\};\quad\{ (\{1,4\},\{2\}),(\{4\},\{2,3\})\};\\
\scriptstyle
\{ (\{1,4\},\{2\}),(\{3\},\{2,4\})\};\quad\{ (\{1,4\},\{2\}),(\{3,4\},\{2\})\};\quad \{ (\{2,4\},\{1\}),(\{3\},\{2,4\})\};\quad\{ (\{2,4\},\{1\}),(\{2,4\},\{3\})\}.
\end{array}$$
\end{example}

%%%%%%%%%%%%%%%%%%%%%%%%%%%
\subsection{Hopf algebra of hypergraphs}
\label{SS:hypergraphHopf}
The acyclic orientations of hypergraphs play an important role in the computation of their antipode
in the Hopf algebra of hypergraphs. This Hopf structure is the image under the Fock functor $\overline{\mathcal K}$
of the Hopf monoid of hypergraphs described in~\cite{BB16}. We recall here what this structure is explicitely.

Given two hypergraphs $G,G'\in {\bf HG[V]}$, we say the $G$ and $G'$ are isomorphic if there exists a permutation $\sigma\colon V\to V$ such that
  $G'=\big\{ \sigma(U)\mid U\in G\big\}$. In this case we write $G\sim G'$.
Let $H$ be the graded vector space
  $$H=\bigoplus_{n\ge 0} H_n=\bigoplus_{n\ge 0} {\mathbb Q} {\bf HG}[n]\big/_{\displaystyle \sim},$$
where $[n]=\{1,2,\ldots,n\}$. That is, for each $n\ge 0$, we consider $H_n={\mathbb Q} {\bf HG}[n]\big/_{\displaystyle \sim}$ the linear span of equivalence classes of hypergraphs on $[n]$.
This space has a structure of graded Hopf algebra given by the following operations.

\medskip
\noindent {\bf Multiplication:} Let $\uparrow^m_n\colon [n]\to \{1+m,\ldots,n+m\}$ be the map that sends $i\in[n]$ to $i+m$.
This induces a map from ${\bf HG}[n]$ to ${\bf HG}[\{1+m,\ldots,n+m\}]$ where
  $$G^{\uparrow^m_n} = \big\{  \{i+m: i\in U\} \mid U\in G \big\}.$$
For all $m,n\ge 0$, we have well defined associative linear operations $\mu_{m,n}\colon H_m\otimes H_n\to H_{m+n}$ given by
  $$\mu_{m,n}(G_1\otimes G_2)= G_1\cup G_2^{\uparrow^m_n},$$
  for $G_1\in {\bf HG}[m]$ and $G_2\in{\bf HG}[n]$. This operation extends to equivalence classes of hypergraphs, and it is commutative since 
    $$ \big(G_1\cup G_2^{\uparrow^m_n}\big) \sim \big( G_2\cup G_1^{\uparrow^n_m}\big).$$
 Thus, $\mu=\sum_{m,n} \mu_{m,n}\colon H\otimes H\to H$ defines a graded, associative, commutative multiplication on $H$.
 The unit $u$ for this operation is given by the unique hypergraph $\emptyset\in {\bf HG}[0]$.

\medskip
\noindent {\bf Comultiplication:} Given $K\subseteq [n]$ let $k=|K|$ and let $St\colon K\to [k]$ be the unique order preserving map between $K$ and $[k]$.
Given a hypergraph $G\in {\bf HG}[n]$ we let
  $$ G\big|_K=\{ U\in G\mid U\subseteq K\} \in {\bf HG}[K]. $$
  We can then use the map $St$ to get a hypergraph $St(G\big|_K)\in {\bf HG}[k]$. For all $m,n\ge 0$, we now have a well defined coassociative linear operations 
$\Delta_{m,n}\colon H_{m+n} \to H_m\otimes H_n$ given by
  $$\Delta_{m,n}(G)= \sum_{K\cup L = [m+n] \atop |K|=m,\ |L|=n} St(G\big|_K)\otimes St(G\big|_L),$$
  for $G\in {\bf HG}[m+n]$. This operation is clearly cocommutative. We have that $\Delta=\sum_{m,n} \Delta_{m,n}\colon H\to H\otimes H$ defines a graded, coassociative, cocommutative comultiplication on $H$.
 The counity for this operation is given by the map $\epsilon\colon H\to{\mathbb Q}$ defined by
   $$\epsilon(G)=\begin{cases} 1&\text{if } G=\emptyset\in{\bf HG}[0],\\ 0&\text{otherwise}.\end{cases}$$

The structure $(H,\mu,u,\Delta,\epsilon)$ gives a structure of graded, connected, commutative and cocommutative bialgebra on $H$. We recall that for such bialgebra
there is a unique antipode $S\colon H\to H$. That gives a structure of graded, connected, commutative and cocommutative Hopf algebra on $H$.

%%%%%%%%%%%%%%%%%%%%%%%%%%%
\subsection{Antipode and acyclic orientations} A \emph{set composition} $A=(A_1,A_2,\ldots,A_k)$ of $I$ is a sequence of nonempty and pairwise disjoint subsets such that $I=A_1\cup A_2\cup\cdots\cup A_k$. We denote this by $A\models I$ and the length $k$ of $A$ is denoted by $\ell(A)$.
One the the subsets $A_i$ is called a \emph{part}.
Similarly, an integer composition $\alpha=(a_1,a_2,\ldots,a_k)$ of $n$ is a sequence of positive integer such that $n=a_1+a_2+\cdots +a_k$. We denote this by $\alpha\models n$ and $k=\ell(\alpha)$
Given a set composition $A\models I$ we get an integer composition using cardinalities: $\alpha(A)=(|A_1|,|A_2|,\ldots,|A_k|)\models |I|$ and $\ell(A)=\ell(\alpha(A))$.

A \emph{set partition} $A = \{A_1, A_2, \ldots, A_k\}$ of $I$ is an unordered collection of nonempty and pairwise disjoint subsets such that $I=A_1\cup A_2\cup\cdots\cup A_k$.
We denote this by $A \vdash I$ and also call each of the subsets $A_i$ a \emph{part} of the partition $A$.

For any graded connected bialgebra $H$ the existence
of the antipode map $S\colon H\to H$ is guaranteed and it can be computed using Takeuchi's formula~\cite{Tak} as follows. For any finite $x\in H_n$
\begin{equation}\label{eq:takeuchi}
 S(x) =\sum_{\alpha\models n} (-1)^{\ell(\alpha)} \mu_\alpha \Delta_\alpha(x)
\end{equation}
Here, for $\ell(\alpha)=1$, we have $\mu_{\alpha}=\Delta_{\alpha}= \text{Id}$ the identity map on $H_n$, and for $\alpha=(a_1,\ldots,a_k)$ with $k>1$,
 $$  \mu_\alpha = \mu_{a_1,n-a_1}(\text{Id}\otimes\mu_{a_2,\ldots,a_k})\qquad\text{and}\qquad
       \Delta_\alpha =(\text{Id}\otimes\Delta_{a_2,\ldots,a_k}) \Delta_{a_1,n-a_1}.
 $$

In the case of hypergraphs, for $G\in {\bf HG[n]}$, the antipode formula gives
  $$ S(G) = \sum_{A\models [n]} (-1)^{\ell(A)} \mu_{\alpha(A)}\big(St(G\big|_{A_1})\otimes  \cdots \otimes St(G\big|_{A_k})\big).$$
But up to a permutation of $[n]$, we have that 
 $$ \mu_{\alpha(A)}\big(St(G\big|_{A_1})\otimes  \cdots \otimes St(G\big|_{A_k})\big)\quad \sim\quad  G\big|_{A_1}\cup G\big|_{A_2}\cup \cdots \cup G\big|_{A_k}.$$  
 We denote the right hand side by $G\big|_{A}=G\big|_{A_1}\cup G\big|_{A_2}\cup \cdots \cup G\big|_{A_k}$ and the antipode formula in this case is
 \begin{equation}\label{eq:anti_hyper}
   S(G) = \sum_{A\models [n]} (-1)^{\ell(A)} G\big|_{A}
   \end{equation}
which contains lots of cancellations. In~\cite{BB16} we give a new formula that involves acyclic orientations of hypergraphs.
To state it we need some notation.

\begin{definition}[Flats] For a hypergraph $G\in{\bf HG}[V]$, given a set composition $A\models V$ we say that $G\big|_{A}$ is a \emph{flat} of $G$. The set 
of all flats of $G$ is denoted by
  $$Flats(G)=\{G\big|_A \,:\, A\models V\}.$$
\end{definition}

Given $G\in{\bf HG}[V]$ and a flat $F\in Flats(G)$, let $A=(A_1,A_2,\ldots,A_k)$ be a finest set composition such that $F=G\big|_A$. Observe
that any permutation of the parts of $A$ gives the same flat $F$ and the set partition $V/F=\{A_1,A_2,\ldots,A_k\}$ is unique and well defined.
We denote by $G/F$ the hypergraph we obtain from $G$ by contracting all the hyperedges in $F$. 
%(More precisely, for any $U=\{i_1,i_2,\ldots i_r\}\subseteq V$
%we let $U/F=\{[i_1],[i_2],\ldots,[i_r]\}\subset V/F$ where $[i]$ denote the equivalent class of $i$ with respect to $V/F$.
%
%Now
%  $$G/F = \big\{ U/F\; : \; U\in G\setminus F \big\}.$$)\caro{I would delete this More precisely statement as it just confuses, mainly since $F$ is a subset of vertices bu definition, and here it mixes as subset of hyperedges. We need to be consistent}
For example,
$$\begin{array}{ccc}
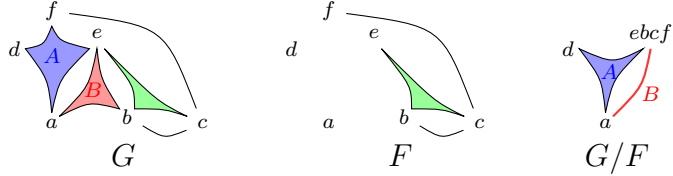
\begin{tikzpicture}[scale=1,baseline=.5cm]
	\node (a) at (0,0) {$\scriptstyle a$};
	\node (b) at (1,.1) {$\scriptstyle b$};
	\node (c) at (2,0) {$\scriptstyle c$};
	\node (d) at (-.5,1) {$\scriptstyle d$};
	\node (e) at (.6,1.2) {$\scriptstyle e$};
	\node (f) at (0,1.5) {$\scriptstyle f$};
	\draw [fill=red!40] (.1,.1) .. controls (.5,.3) .. (.9,.2) .. controls (.65,.6) .. (.6,1) .. controls (.5,.5) .. (.1,.1) ; 
	\draw [fill=blue!40] (0,.15) .. controls (-.05,.6) .. (-.35,1) .. controls (-.1,1.05) .. (0,1.3)  .. controls (.1,1.1) .. (.5,1) .. controls (.1,.6) .. (0,.15) ; 
	\draw [fill=green!40] (1.8,.1) .. controls (1.5,.2) .. (1.1,.2) .. controls (1.1,.5) .. (.7,1) .. controls (1.4,.3) .. (1.8,.1) ; 
	\draw (b) .. controls (1.5,-.2) ..   (c); 
	\draw (f) .. controls (1.5,1.3) ..   (c); 
	\node at (0,.9) {$\scriptstyle \blue{A}$};
	\node at (.55,.45) {$\scriptstyle \red{B}$};
\end{tikzpicture}  \quad & 
\begin{tikzpicture}[scale=1,baseline=.5cm]
	\node (a) at (0,0) {$\scriptstyle a$};
	\node (b) at (1,.1) {$\scriptstyle b$};
	\node (c) at (2,0) {$\scriptstyle c$};
	\node (d) at (-.5,1) {$\scriptstyle d$};
	\node (e) at (.6,1.2) {$\scriptstyle e$};
	\node (f) at (0,1.5) {$\scriptstyle f$};
	\draw [fill=green!40] (1.8,.1) .. controls (1.5,.2) .. (1.1,.2) .. controls (1.1,.5) .. (.7,1) .. controls (1.4,.3) .. (1.8,.1) ; 
	\draw (b) .. controls (1.5,-.2) ..   (c); 
	\draw (f) .. controls (1.5,1.3) ..   (c); 
\end{tikzpicture}  \quad & 
\begin{tikzpicture}[scale=1,baseline=.5cm]
	\node (a) at (0,0) {$\scriptstyle a$};
	\node (d) at (-.5,1) {$\scriptstyle d$};
	\node (e) at (.6,1.2) {$\scriptstyle ebcf$};
	\draw [thick,color=red!80] (.6,1) .. controls (.5,.5) .. (.1,.1) ; 
	\draw [fill=blue!40] (0,.15) .. controls (-.05,.6) .. (-.35,1) .. controls (.1,.8)  .. (.5,1) .. controls (.1,.6) .. (0,.15) ; 
	\node at (0.05,.7) {$\scriptstyle \blue{A}$};
	\node at (.6,.4) {$\scriptstyle \red{B}$};
\end{tikzpicture}\\
G & F & G/F \\
\end{array}
$$
Given an orientation  ${\mathcal O}$ of $G/F$, denote by $V/ {\mathcal O}$ the set partition of $V$ obtained from
 the set partition $(V/F)/{\mathcal O}$, where the parts of $V/F$ are put together according to $(V/F)/{\mathcal O}$. For a hypergraph $G\in{\bf HG}[V]$, let $\mathfrak O(G)$ denote the set of all its acyclic orientations. 
%For ${\mathcal O}\in {\mathfrak O}(G)$, we denote by $|G/{\mathcal O}|$ the number of distinct equivalence classes induced on $V$  by the orientation $\mathcal O$. 
We now extend~\cite[Lemma 3.13]{BB16} to all set compositions. 
For $A=(A_1,A_2,\ldots,A_k)\models[n]$, and every $1\le i\le k$, let $A_{i,k}=A_i\cup A_{i+1}\cup\cdots\cup A_k$ and let 
    $G/\mathcal{O}_{i,k}=(G/\mathcal{O})\big|_{A_{i,k}}$ where $\mathcal O$ is an orientation of $G$.
    
\begin{remark}
The symbol $/$ is overloaded.
Its meaning will be clear from the context, but we warn the reader that the meaning of the symbol $/$ depends on what type of objects are involved.
\end{remark}

\begin{lemma}\label{lem:Omega}
Fix $G\in{\bf HG}[n]$. There is a surjection $\Omega$ and injection $\Psi$
 $$ \begin{tikzpicture}[scale=1,baseline=0cm]
	\node at (0,0) {$\{A\mid\   A\models [n]\}$};
	\node at (6.2,-.25) {$\displaystyle \bigcup_{F\in Flats(G)} {\mathfrak O}(G/F)\,,$};
	\draw [->>] (1.5,.1)--node[above=.1]{$\scriptstyle \Omega $} (4.3,.1);
	\draw [<-right hook] (1.5,-.1)--node[below=.1]{$\scriptstyle \Psi $} (4.3,-.1);
       \end{tikzpicture}
  $$
where the maps $\Omega$ and $\Psi$ depend on $G$ and are obtained as follows:
%\caro{to check: use of same notation for set partition and set composition}
\begin{enumerate}
 \item[(a)] For  $A=(A_1,A_2,\ldots, A_k)\models [n]$ we let $F=G\big|_A$. For each hyperedge $U\in G/F$ let $i=\min\{j: A_j\cap U\ne \emptyset\}$ then $(U\cap A_i,U- A_i)$ defines an acyclic orientation for each $U$ and it gives $\Omega(A)\in {\mathfrak O}(G/F)$.
 Furthermore $[n]/{\Omega(A)}$ is a refinement of $\{A_1,A_2,\ldots,A_k\}$. 
 \item[(b)] 
 For ${\mathcal O}\in {\mathfrak O}(G/F)$, let $\Psi(\mathcal O)=(A_1,A_2,\ldots, A_k)\models[n]$ be such that  and $A_i$ is the unique source of the restriction  $G/\mathcal{O}_{i,k}$  where $\min(A_i)$ is maximal among the sources of $G/\mathcal{O}_{i,k}$.
 Here a source is any vertex with no incoming edges.
\end{enumerate}
Also, we have that \,$\Omega\circ\Psi =\rm{Id}$ and for $(A_1,A_2,\ldots, A_k)=\Psi(\mathcal O)$ it follows that $[n]/{\mathcal O}=\{A_1,A_2,\ldots, A_k\}$.
\end{lemma}

\begin{proof}
Set $V = \{1,2,\dots,n\}$.
   For (a), let  $A=(A_1,A_2,\ldots,A_k)\models V$ and $F=G\big|_A$.  For any $U\in G/F$, we always have $A_i\cap U\ne U$.
   Hence $(U\cap A_i,U\setminus A_i)$ for $i=\min\{j: A_j\cap U\ne \emptyset\}$ defines a proper orientation $\mathcal O$ of  $G/F$. 
   By construction, each head $\mathfrak{a}$ for $(\mathfrak{a}, \mathfrak{b}) \in \mathcal O$ is completely included within a part $A_i$ for a unique part $1\le i\le k$. 
   This implies that $V/{\mathcal O}$   refines $\{A_1,\ldots,A_k\}$ and it allows us to define a function $f\colon V/{\mathcal O}\to \{1,2,\ldots,k\}$ where $f([v])=i$ if and only if $[v]\subseteq A_i$.
   By construction of $\mathcal O$, for any $([{\mathfrak a}],[b])\in (G/F)/{\mathcal O}$ the function $f$ is such that $f([\mathfrak{a}])<f([b])$.
   Hence  $(G/F)/{\mathcal O}$ has no cycles and $\Omega(A)=\mathcal O$ is a well defined acyclic orientation of $G/F$.
   
   For  (b), let $\mathcal O$ be an acyclic orientation on $G/F$. Let us show that the set composition $\Psi({\mathcal O})=(A_1,\ldots, A_k)$ is well defined in (b).
   That is,  we will show how to construct the only possible set composition $(A_1,\ldots, A_k)$ satisfying the conditions of (b).
   Recall the vertices of $(G/F)/\mathcal{O}$ are equivalences class, and hence subsets of $V$.
   Let us consider the partial order on subsets of $V$ by $A < B$ whenever $\min (A) < \min (B)$ for $A,B \subseteq V$.
   This is a partial order on the subsets of $V$, but is a total order on the vertices of $(G/F)/\mathcal{O}$ since the vertices of $(G/F)/\mathcal{O}$ consists of a collection of disjoint subsets of $V$.
Set $G_1 = (G/F)/\mathcal{O}$. Given that $\mathcal{O}$ is an acyclic orientation then the directed graph $G_1$ must have a source.
Moreover, if we remove any collection of vertices from $G_1$, the remaining graph  still has a source.
If $(A_1,\ldots, A_k)$ is any set composition satisfying (b), then $A_i$ must be the largest source of $G_i$.
Thus the set composition $(A_1,\ldots, A_k)$ exists and is well defined.
It is clear from this realization that $\{A_1,\ldots, A_k\}=V/{\mathcal O}$.
   
 We now need to show that $\Omega\circ\Psi =\rm{Id}$. 
 For any $(\mathfrak{a},\mathfrak{b})\in {\mathcal O}$ we must have $\mathfrak{a}\subseteq A_i$ for some unique $1\le i\le k$. 
   We claim that
 $$A_j\cap   \mathfrak{b}\ne \emptyset\ \quad\implies \quad j>i$$
 If not, then there would be $j<i$ such that $A_j\cap \mathfrak{b}\ne \emptyset$.  This means that there is an edge from $A_i$ to $A_j$ in $G/{\mathcal O}_{j,k}$, which contradicts the fact that $A_j$ is a source of $G/{\mathcal O}_{j,k}$, hence $j$ must be such that $j>i$. Therefore $\Omega(\Psi({\mathcal O}))={\mathcal O}$.
\end{proof}

\begin{theorem}[\cite{BB16}, Theorem 3.16] \label{thm:bb16}
For $G\in{\bf HG}[n]$,
   $$S(G) = \sum_{F\in Flats(G)} a(G/F) F\,,\qquad\text{where}\qquad a(G/F)=\sum_{{\mathcal O} \in{\mathfrak O}(G/F)} (-1)^{|\Psi({\mathcal O})|}\,.$$
\end{theorem}

\begin{proof}[Sketch of proof] A full proof is given in~\cite{BB16}. We sketch here a slightly different proof.
For $G\in{\bf HG}[n]$ and $F\in Flats(G)$, let ${\mathcal C}_G^F=\{A\models[n]\, : \, G\big|_A=F\}$. 
Starting from~\eqref{eq:anti_hyper} we have
$$
 S(G) = \sum_{A\models [n]} (-1)^{\ell(A)} G\big|_{A} = \sum_{F\in Flats(G)} \Big( \sum_{{\mathcal O} \in{\mathfrak O}(G/F)}
           \big( \sum_{A\in {\mathcal C}_G^F \atop \Omega(A)={\mathcal O}}  (-1)^{\ell(A)} \big) \Big) F
 $$
 The sign reversing involution in~\cite[Theorem 3.16]{BB16} gives
 $$ \sum_{A\in {\mathcal C}_G^F \atop \Omega(A)={\mathcal O}}  (-1)^{\ell(A)}  = (-1)^{|\Psi({\mathcal O})|}$$
 and this gives us the desired result.
\end{proof}

%%%%%%%%%%%%%%%%%%%%%%%%%%%
\subsection{Hypergraphic polytope}
One of our main goals is to give a geometric meaning to Theorem~\ref{thm:bb16}. One of the beautiful results in~\cite[Corollary 13.7]{Aguiar-Ardila} shows that the antipode of a simple graph can be recovered from the faces of its graphical zonotope. They also give a geometric interpretation~\cite[Corollary 21.3]{Aguiar-Ardila} for the antipode in $\mathcal A$ of simplicial complexes (see~\cite{BenedettiHallamMachacek}), which is an interpretation that was noted independently by the first author. 
The geometric object behind Theorem~\ref{thm:bb16} is the {hypergraphic polytope}.
We let $\{e_1,\dots,e_n\}$ denote the standard basis of $\mathbb R^n$.

\begin{definition} 
 Given a hypergraph $G\in{\bf HG}[n]$,
  the \emph{hypergraphic polytope} $P_G$ associated to $G$ is the polytope in ${\mathbb R}^n$
  defined by the Minkowski sum
      $$P_G=\sum_{U\in G} {\bf \Delta}_U,$$
      where ${\bf \Delta}_U$ is the simplex given by the convex hull of the points $\{e_i \mid i\in U\}$.
\end{definition}

So, a hypergraphic polytope is a Minkowski sum of standard simplicies.
We note that we are aware of such Minkowski sums being previously studied by in~\cite{AgnMorris, Agn, Ag16}.
We will consider a particular example of a hypergraphic polytope defined by Agnarsson~\cite{Ag16} call the hyper-permutahedra in Section~\ref{ss:hyper-perm}.

\begin{example}\label{ex:PG1}
Consider the hypergraph
$G=
\begin{tikzpicture}[scale=.7,baseline=.2cm]
	\node (1) at (0,.1) {$\scriptstyle 1$};
	\node (2) at (0,1) {$\scriptstyle 3$};
	\node (3) at (1,.5) {$\scriptstyle 2$};
	\draw [fill=blue!40] (.1,.1) .. controls (.25,.5) .. (.15,.9) .. controls (.4,.5) .. (.85,.5) .. controls (.5,.4) .. (.1,.1) ; 
	\draw [color=red!80,thick] (.9,.7).. controls (0.6,1) ..(0.2,1);
\end{tikzpicture} 
$.
We  have
$$\begin{array}{ccc}
 \begin{tikzpicture}[scale=1,baseline=.5cm]
	\node (1) at (0.6,1.0) {$\scriptstyle e_1$};
	\node (2) at (-.2,-.2) {$\scriptstyle e_2$};
	\node (3) at (1.5,-.2) {$\scriptstyle e_3$};
	\draw [fill=blue!40] (0,0) -- (.6,.75) -- (1.2,0) --(0,0) ; 
\end{tikzpicture} \quad &
 \begin{tikzpicture}[scale=1,baseline=.5cm]
	\node at (-.2,0) {$\scriptstyle e_2$};
	\node at (1.5,0) {$\scriptstyle e_3$};
	\draw [thick,color=red] (0,0) -- (1.2,0); 
\end{tikzpicture} \quad &
\begin{tikzpicture}[scale=1,baseline=.5cm]
	\draw [fill=gray!10] (0,0) -- (.6,.75) -- (1.8,.75) -- (2.4,0) --(0,0) ; 
	\draw [color=gray!10,fill=blue!20] (0,0) -- (.6,.75) -- (1.2,0) --(0,0) ; 
	\draw (0,0) -- (.6,.75) -- (1.8,.75) -- (2.4,0) --(0,0) ; 
	\draw [thick,color=red!80] (1.2,0) -- (2.4,0); 
\end{tikzpicture}\\
\blue{{\bf \Delta}_{123}}& \red{{\bf \Delta}_{23}} & P_G=\blue{{\bf \Delta}_{123}}+ \red{{\bf \Delta}_{23}}\\
\end{array}
$$
\end{example}

\begin{example}\label{ex:PG2}
For the hypergraph
$G'=
\begin{tikzpicture}[scale=.7,baseline=.2cm]
	\node (1) at (0,.1) {$\scriptstyle 1$};
	\node (2) at (2.5,.5) {$\scriptstyle 4$};
	\node (3) at (0,1) {$\scriptstyle 2$};
	\node (4) at (1,.5) {$\scriptstyle 3$};
	\draw [fill=blue!40] (.1,.1) .. controls (.25,.5) .. (.15,.9) .. controls (.4,.5) .. (.85,.5) .. controls (.5,.4) .. (.1,.1) ; 
	\draw [thick,color=red!80] (4).. controls (1.8,.6) ..(2);
\end{tikzpicture} 
$,
we  have
$$\begin{array}{ccc}
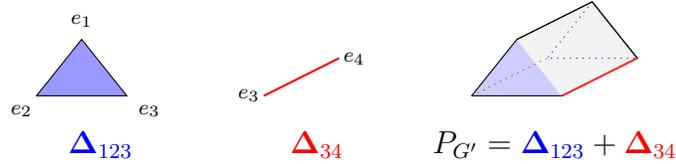
 \begin{tikzpicture}[scale=1,baseline=.5cm]
	\node (1) at (0.6,1.0) {$\scriptstyle e_1$};
	\node (2) at (-.2,-.2) {$\scriptstyle e_2$};
	\node (3) at (1.5,-.2) {$\scriptstyle e_3$};
	\draw [fill=blue!40] (0,0) -- (.6,.75) -- (1.2,0) --(0,0) ; 
\end{tikzpicture} \quad &
 \begin{tikzpicture}[scale=1,baseline=.5cm]
	\node at (-.2,0) {$\scriptstyle e_3$};
	\node at (1.2,.5) {$\scriptstyle e_4$};
	\draw [thick,color=red] (0,0) -- (1,.5 ); 
\end{tikzpicture} \quad &
\begin{tikzpicture}[scale=1,baseline=.5cm]
	\draw [fill=gray!10] (0,0) -- (.6,.75)-- (1.6,1.25) -- (2.2,.5) -- (1.2,0) --(0,0) ; 
	\draw [color=gray!10,fill=blue!20] (0,0) -- (.6,.75) -- (1.2,0) --(0,0) ; 
	\draw [dotted,color=red] (0,0)--(1,.5);
	\draw [dotted,color=blue] (1.6,1.25)--(1,.5)--(2.2,.5);
	\draw (0,0) -- (.6,.75)-- (1.6,1.25) -- (2.2,.5) -- (1.2,0) -- (0,0) ; 
	\draw [thick,color=red!80] (1.2,0) -- (2.2,0.5); 
\end{tikzpicture}\\
\blue{{\bf \Delta}_{123}}& \red{{\bf \Delta}_{34}} & P_{G'}=\blue{{\bf \Delta}_{123}}+ \red{{\bf \Delta}_{34}}\\
\end{array}
$$
which is a 3-dimensional polytope.
\end{example}

We want to get a good description of the normal fan ${\mathcal N}(P_G)$ of the hypergraphic polytope $P_G$.
We refer the reader to~\cite[Chapter 7]{Z95} for more details and notation about normal fans.
First let us describe the normal fan of a simplex. Given a linear functional $\mathbf x:\mathbb R^n\rightarrow\mathbb R$  we will identify $\mathbf x$ with the vector $(x_1,\dots,x_n)$ where $x_i:=\mathbf x(e_i)$. In this way, if $a=\sum_{i=1}^na_i e_i\in\mathbb R^n$, we have that $\mathbf x(a) = (a_1,\dots,a_n)\cdot(x_1,\dots,x_n)$. Now, notice that the faces of the simplex ${\bf\Delta}_U$ are in bijection with the nonempty subsets $K\subseteq U$. Thus, each cone in $\mathcal N({\bf\Delta}_U)$ is also indexed by such $K$. Moreover,

\begin{lemma}\label{lem:fansimplex}
 Let $U\subseteq [n]$ with $|U|=r\ge 2$. For $\emptyset\ne K\subseteq U$, the cone $C_{K, U}$ in $\mathcal N({\bf \Delta}_U)$ corresponding to the face ${\bf\Delta}_K$ of ${\bf\Delta}_U$
is given by
$$ C_{K, U}:=\{\mathbf x\in(\mathbb R^n)^* \mid x_i=x_j \text{ for } i,j\in K;\ x_i\ge x_j \text{ for } i\in K \text{ and } j\in U\setminus K\}.$$
\end{lemma}

\begin{proof}
The vertices in the face ${\bf \Delta}_K$ are $\{e_a\,:\, a\in K\}$. Thus the linear functionals $\mathbf x$ attaining their maximum at this face are precisely those described by $C_{K,U}$.

\end{proof}

\begin{remark} The nontrivial faces ${\bf \Delta}_K$ of ${\bf \Delta}_U$ are in bijection with the orientations of the hyperedge $U$. For instance, taking $U_1=\{a,b\}$ and $U_2=\{a,b,c\}$ gives us the following labeling of the corresponding faces 
$$\begin{array}{cc}
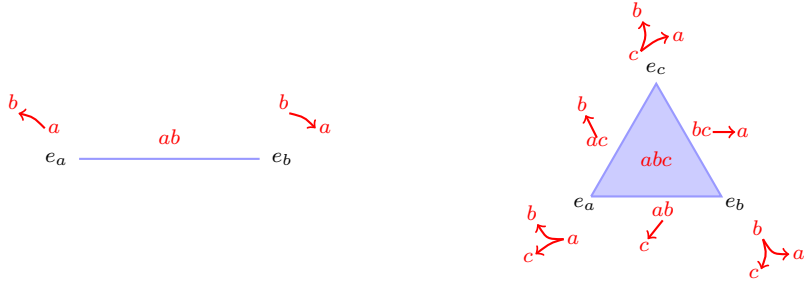
 \begin{tikzpicture}[scale=1.5,baseline=.0cm]
	\node (a) at (-1,0) {$\scriptstyle e_a$};
	\node (b) at (1,0) {$\scriptstyle e_b$};
	\node at (0,.2) {\red{$\scriptstyle ab$}};
	\node at (-1.2,.4) {\red{$\begin{tikzpicture}[scale=.5,baseline=.5cm]
		\node (2) at (0,1.2) {$\scriptstyle b$};
		\node (3) at (1.1,.5) {$\scriptstyle a$};
		\draw [thick,->] (.85,.5) .. controls (.55,.8) .. (.15,.9) ; 
		\end{tikzpicture} $}};
	\node at (1.2,.4) {\red{$\begin{tikzpicture}[scale=.5,baseline=.5cm]
		\node (2) at (0,1.2) {$\scriptstyle b$};
		\node (3) at (1.1,.5) {$\scriptstyle a$};
		\draw [thick,<-] (.85,.5) .. controls (.55,.8) .. (.15,.9) ; 
		\end{tikzpicture} $}};
	\draw [color=blue!40,thick] (-.8,0)--(.8,0) ; 
\end{tikzpicture}\qquad&\qquad
\begin{tikzpicture}[scale=1,baseline=.0cm]
	\node (a) at (-.95,-.6) {$\scriptstyle e_a$};
	\node (b) at (1.05,-.6) {$\scriptstyle e_b$};
	\node (c) at (0,1.2) {$\scriptstyle e_c$};
	\node at (0,-.9) {\red{$\begin{tikzpicture}[scale=.5,baseline=.5cm]
		\node (1) at (0,0) {$\scriptstyle c$};
		\node (3) at (.5,1) {$\scriptstyle ab$};
		\draw [thick,->] (.5,.7) --(.1,.2) ; 
		\end{tikzpicture} $}};
	\node at (.85,.4) {\red{$\begin{tikzpicture}[scale=.5,baseline=.5cm]
		\node (1) at (-.1,.6) {$\scriptstyle bc$};
		\node (3) at (1,.5) {$\scriptstyle a$};
		\draw [thick,->] (.2,.5) -- (.8,.5) ; 
		\end{tikzpicture} $}};
	\node at (-.85,.5) {\red{$\begin{tikzpicture}[scale=.5,baseline=.5cm]
		\node (1) at (.5,0) {$\scriptstyle ac$};
		\node (3) at (.1,1) {$\scriptstyle b$};
		\draw [thick,->] (.5,.1) -- (.2,.7) ; 
		\end{tikzpicture} $}};
	\node at (0,1.7) {\red{$\begin{tikzpicture}[scale=.5,baseline=.5cm]
		\node (1) at (-.1,0) {$\scriptstyle c$};
		\node (2) at (0,1.2) {$\scriptstyle b$};
		\node (3) at (1.1,.5) {$\scriptstyle a$};
		\draw [thick,->] (.1,.1) .. controls (.25,.5) .. (.15,.9) ; 
		\draw [thick,->] (.1,.1) .. controls (.5,.4) .. (.85,.5) ; 
		\end{tikzpicture} $}};
	\node at (1.6,-1.2) {\red{$\begin{tikzpicture}[scale=.5,baseline=.5cm]
		\node (1) at (-.1,0) {$\scriptstyle c$};
		\node (2) at (0,1.2) {$\scriptstyle b$};
		\node (3) at (1.1,.5) {$\scriptstyle a$};
		\draw [thick,->] (.15,.9) .. controls (.4,.5) .. (.85,.5) ; 
		\draw [thick,->] (.15,.9) .. controls (.25,.5) .. (.1,.1) ; 
		\end{tikzpicture} $}};
	\node at (-1.4,-1) {\red{$\begin{tikzpicture}[scale=.5,baseline=.5cm]
		\node (1) at (-.1,0) {$\scriptstyle c$};
		\node (2) at (0,1.2) {$\scriptstyle b$};
		\node (3) at (1.1,.5) {$\scriptstyle a$};
		\draw [thick,->] (.85,.5) .. controls (.4,.5) .. (.15,.9) ; 
		\draw [thick,->] (.85,.5) .. controls (.5,.4) .. (.1,.1) ; 
		\end{tikzpicture} $}};
	\draw [color=blue!40,thick,fill=blue!20] (-.866,-.5)--(.866,-.5)--(0,1)--(-.866,-.5) ; 
	\node at (0,0) {\red{$\scriptstyle abc$}};
\end{tikzpicture}
\\
U_1=\{a,b\},&U_2=\{a,b,c\}.\\
\end{array}
$$
This allows to think of inequalities describing cones in terms of orientations.
For example \red{$\begin{tikzpicture}[scale=.5,baseline=.1cm]
		\node (1) at (-.1,.6) {$\scriptstyle bc$};
		\node (3) at (1.3,.6) {$\scriptstyle ad$};
		\draw [thick,->] (.2,.5) -- (.8,.5) ; 
		\end{tikzpicture} $} corresponds to \red{$x_b=x_c\ge x_a=x_d$}.
The interior of ${\bf \Delta}_U$ corresponds to the contraction of the hyperedge $U$.
\end{remark}
% Proposition 7.12 say that the Normal fan of sum is meet of the fans

We are now ready to state and prove the main theorem of this section.

%%%%%%%%%%%%%%%%%%%%%%%%%%%
\subsection{Main Theorem:}

Let $G\in{\bf HG}[n]$ and $P_G$ its hypergraphic polytope. We now show that the faces of $P_G$ are naturally
labeled by the acyclic orientations of the contractions $G/F$ for each flat $F$ of $G$. For that purpose we introduce some more notation. Let $\mathcal O\in {\mathfrak O}(G/F)$ and define the cone $C_{\mathcal O}$ by
 $$C_{\mathcal O} :=\Big\{\mathbf x \in ({\mathbb R}^n)^* \,\Big|\, {\red{x_a=x_b}\quad\text{\small if $a,b$ are identified in $[n]/{\mathcal O}$}
                                                                                 \atop \red{x_a\ge x_b} \quad\text{\small if $([a],[b])$ is an arrow of $(G/F)/{\mathcal O}$}\hfill }  \Big\}\,.$$
                                                                                 
\begin{remark}
It follows immediately from the definition of $C_{\mathcal{O}}$ that $\dim C_{\mathcal{O}} = |[n] / \mathcal{O}|$.
This equality is straightforward, but we will find it to be a useful fact. 
\end{remark}
                                           
\begin{remark}

The cones $C_{\O}$ are present in the cone-preposet dictionary of Postnikov, Reiner, and Williams~\cite{PRW}.
The relationship of our results with the cone-preposet dictionary is elaborated on in Section~\ref{S:polytope}.
Also in the computer science community the term \emph{weak ordering} is used to referred to such orientations.
\end{remark}

\begin{theorem}\label{THM:HypegraphicPolytope}
Given $G\in{\bf HG}[n]$, the normal fan ${\mathcal N}(P_G)$ of $P_G$ in $({\mathbb R}^n)^*$ is defined by the cones $C_{\mathcal O}$
where $\mathcal O$ runs over the set $\mathcal {AO}=\displaystyle \bigcup_{F\in Flats(G)}{\mathfrak O}(G/F)$.
%acyclic orientations of contractions of $G$ by flats. More precisely for $F\in Flats(G)$ and ${\mathcal O}\in {\mathfrak O}(G/F)$
%we have 
%  $$C_{\mathcal O} =\Big\{ X\in {\mathbb R}^n \,\Big|\, {\red{x_a=x_b}\quad\text{\small if $a,b$ are identified in $(G/F)/{\mathcal O}$}
%                                                                                 \atop \red{x_a\ge x_b} \quad\text{\small if $([a],[b])$ is an arrow of $(G/F)/{\mathcal O}$}\hfill }  \Big\}\,.$$
 In particular, the faces of $P_G$ are in one to one correspondence with the elements $\mathcal O\in\mathcal {AO}$.
 %   $\displaystyle \bigcup_{F\in Flats(G)}{\mathfrak O}(G/F)\,.$                                                            
\end{theorem}

\begin{proof} 
First we show that for a given $F\in Flats(G)$ and $\mathcal O\in{\mathfrak O}(G/F)$, the cone $C_{\mathcal O}$ is a cone in ${\mathcal N}(P_G)$. Since $P_G=\sum_{U\in G} {\bf \Delta}_U$, Proposition 7.12 of~\cite{Z95} 
tells us that
$ {\mathcal N}(P_G)=\bigwedge_{U\in G} {\mathcal N}({\bf \Delta}_U)$.
Here $\bigwedge$ denotes the common refinement of fans.
 
Let  $A=(A_1,\ldots,A_k)=\Psi({\mathcal O})$
be given by Lemma~\ref{lem:Omega}.
For any $U\in G$ define $K(U)$ to be $K(U)= U\cap A_i$ where $i$ is the minimal index with the property that $U \cap A_i \neq\emptyset$. 
%Thus $K(U)$ is the head of the orientation of the hyperedge $U$ in $\mathcal O$. 
When $K(U)\ne U$ it is the head of an edge of $\mathcal O$, when $K(U)=U$ the hyperedge $U$ is contracted in $G/F$.
We have that 
	$$ C_{\mathcal O}=\bigcap_{U\in G} C_{K(U), U}$$
	is a cone of ${\mathcal N}(P_G)$.
 
%Conversely, suppose  $C$ is a cone of  $ {\mathcal N}(P_G)$. Using Lemma~\ref{lem:fansimplex} 
%  $$C=\bigcap_{U\in G} C_{K(U)}$$ for some $\emptyset\ne K(U)\subseteq U$.
 %We have that either $K(U)\subset U$ or $K(U)=U$.
%In the case $K(U)\subset U$, we get an orientation of $U$, and in the case $K(U)=U$, we have that $U$ is contracted. 
%The sequence of contractions defines a single flat $F'=G\big|_A$, given by any set composition $A=(A_1,A_2,\ldots,A_r)$ corresponding to the equivalence classes of the $U$'s that are contracted.
%Now, we have an orientation 
%     $${\mathcal O}'=\big\{ \big(K(U),U\setminus K(U)\big) \,\big|\, U\in G,\, K(U)\ne U  \big\}$$      of $G/F'$, potentially with loops. But each cycle $A_{i_1}\to A_{i_2}\to \cdots\to A_{i_s}=A_{i_1}$ implies the following sequence of inequalities on $C$
%\begin{equation}\label{eq:Ccycles} x_{a_1}=x_{a_2}=\cdots=x_{a_{m_1}}\ge x_{a_{m_1+1}}=\cdots= x_{a_{m_2}}\ge \cdots \ge x_{a_{m_{s-1}+1}}=x_{a_1},
%\end{equation} where $A_{i_1}=\{a_1,a_2,\ldots,a_{m_1}\}$, $A_2=\{ a_{m_1+1},\ldots,a_{m_2} \}$, $\ldots$, $A_{i_{s}}=A_{1_1}$. Which means that in $C$ we must have equality of all the coordinates in~\eqref{eq:Ccycles}. Let $B=(B_1,B_1,\ldots,B_m)$ be any set composition obtained from $A$ by merging together any cycle of $(G/F')/{\mathcal O}'$.  The set composition $B$ has the property that for $U\in G$ we have $U\subseteq B_i$ for some $1\le i\le m$  if and only if  $C\subseteq  C_{U}.$

For the converse, let $C=\bigcap_{U\in G} C_{K(U), U}$ be a cone in $\mathcal N(P_G)$.
Now $K(U)$ is some arbitrary nonempty subset of $U$ rather than the particular subset from the first part of this proof.
In this manner, we can think of such $C$ as a family $\{(U,K(U))\; : \;U\in G\}$. This description of $C$ is not unique. We will construct via the following algorithm an orientation $\mathcal O\in\mathcal {AO}$ such that $C=C_{\mathcal O}$.

\begin{enumerate}   
\item[{\bf (1)}] (input) A family $\{(U,K(U))\; : \;U\in G\}$ such that $C=\bigcap_{U\in G} C_{K(U), U}$.
\item[{\bf (2)}] (construct flat $F$) In the above description contract every hyperedge $U$ such that  $K(U)=U$. This defines a flat $F$ of $G$ which contains all hyperedges $U$ such that $K(U)=U$.
\item[{\bf (3)}] For every subset $A\subseteq [n]$, let $\overline A$ denote the image of $A$ in $[n]/F$. If there is $U$ for which $K(U)\neq U$ and $\overline{K(U)}=\overline U$ then set $K(U)=U$ and go back to (2).  
\item[{\bf (4)}] (define orientation of $G/F$) At this step, for each $U$ such that $|\overline U|>1$,  we have that $\overline{K(U)}\neq\overline U$. These $\overline{K(U)}$ define an orientation
${\mathcal O}$ of $G/F$.
\item[{\bf (5)}] (resolve cycles) 
If $(G/F)/{\mathcal O}$ has a cycle $A_{i_1}\to A_{i_2}\to \cdots\to A_{i_s}=A_{i_1}$ where $A_{i_j}=\{a_{j,1},a_{j,2},\ldots,a_{j,m_j}\}$, then we have the following relations in $C$
   $$x_{a_{1,1}}=x_{a_{1,2}}=\cdots=x_{a_{1,m_1}}\ge x_{a_{2,1}}=\cdots= x_{a_{2,m_2}}\ge \cdots \ge x_{a_{s,1}}=x_{a_{1,1}}.$$
This implies that all the coordinates indexed by $B=A_{i_1}\cup A_{i_2}\cup \cdots\cup A_{i_s}$ are equal in $C$. Set $K(U)=K(U)\cup(B\cap U)$ whenever $K(U)\cap B\neq\emptyset$. Go back to step (2).
\item[{\bf (6)}] (output $\mathcal O$) The orientation $\mathcal O$ of $G/F$ which is acyclic and $C=C_{\mathcal O}$.
\end{enumerate} 
To finish the proof we notice that the algorithm stops and that at all steps $C=\bigcap_{U\in G} C_{K(U),U}$. This follows since in the algorithm the family $\{(U,K(U))\; : \;U\in G\}$ is modified only in steps (3) and (5).
Each modification only increases the sets $K(U)$ for some $U$. Since $G$ is finite, the algorithm must stop. When it stops, the orientation $\mathcal O$ has no cycles, thanks to (5) where any edge that is part of a cycle is contracted to a single point.
On the other hand in the starting point, the sets $K(U)$ give us that 
  $$i,j\in K(U) \text{ for some } U \quad\implies\quad x_i=x_j \text { in }C.$$
The equivalence relation $[n]/F$ in step (3) is such that if $i\sim j$ in $[n]/F$, then $x_i=x_j$ in $C$.
Hence, in step (3), if $\overline{K(U)}=\overline U$,  for all $i,j\in U$ we have $x_i=x_j$ in $C$. This implies that if we redefine $K(U)=U$ we do not change the cone $C$.
Similarly, in step (5), we have shown that for any $i,j\in B$ we have $x_i=x_j$ in $C$. Hence if $K(U)\cap B\neq\emptyset$, redefining $K(U)=K(U)\cup(B\cap U)$ does not change $C$.
We have shown that the algorithm preserves the cone $C$ and produces the desired orientation.
\end{proof}

\begin{example} \label{ex:NPG1} Consider the hypergraph $G$ in Example~\ref{ex:PG1} and $P_G=\blue{{\bf \Delta}_{123}}+ \red{{\bf \Delta}_{23}}$.
The normal fan of $P_G$ has 9 cones. It is the common refinement of the normal fans of $\blue{{\bf \Delta}_{123}}$ and $\red{{\bf \Delta}_{23}}$.
 $$ 
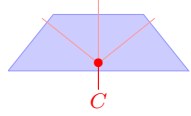
\begin{tikzpicture}[scale=1,baseline=.5cm]
	\draw [color=blue!40, fill=blue!20] (0,0) -- (.6,.75) -- (1.8,.75) -- (2.4,0) --(0,0) ; 
	\draw [color=red!40] (1.2,.1)--(1.95,.7);
	\draw [color=red!40] (1.2,.1)--(.45,.7);
	\draw [color=red!40] (1.2,1)--(1.2,-.25);
	\draw [color=red] (1.2,.1)--(1.2,-.25);
	\node at (1.2,.1) {\red{$\scriptstyle \bullet$}};
	\node at (1.2,-.4) {\red{$\scriptstyle C$}};
\end{tikzpicture}
$$

Take the cone $C=\{X\mid x_2=x_3>x_1\}$ of $\mathcal N(P_G)$.
It can be obtained as an intersection given by $\{(U,K(U)):U\in G\}$ in different ways.
We can consider inputting the family $\big\{\blue{(\{1,2,3\},\{2,3\})},\red{(\{2,3\},\{2,3\})}\big\}$ describing $C$ into the algorithm from the proof.
In step (2) the algorithm will construct the flat consisting of the hyperedge $\{2,3\}$ and then will output the acyclic orientation
$\begin{tikzpicture}[scale=.5,baseline=.3cm]
		\node (1) at (0,.1) {$\scriptstyle 1$};
		\node (3) at (.5,1) {$\scriptstyle 23$};
		\blue{\draw [->] (.5,.7) --(.1,.2) ; }
		\end{tikzpicture} $.
Instead, if we start the algorithm the family $\big\{\blue{(\{1,2,3\},\{3\})},\red{(\{2,3\},\{2\})}\big\}$ describing the same $C$, we will construct the empty flat in step (2) and step (5) gives us 
$$\begin{tikzpicture}[scale=.5,baseline=.3cm]
		\node (1) at (0,.1) {$\scriptstyle 1$};
		\node (2) at (0,1.2) {$\scriptstyle 3$};
		\node (3) at (1,.5) {$\scriptstyle 2$};
		\blue{\draw [->] (.15,.9) .. controls (.4,.5) .. (.85,.5) ; 
		\draw [->] (.15,.9) .. controls (.25,.5) .. (.1,.1) ; }
		\end{tikzpicture} \bigcap 
\begin{tikzpicture}[scale=.5,baseline=.3cm]
		\node (2) at (0,1.2) {$\scriptstyle 3$};
		\node (3) at (1,.5) {$\scriptstyle 2$};
		\red{\draw [->] (.85,.5) .. controls (.55,.8) .. (.15,.9) ; }
		\end{tikzpicture}  = 
\begin{tikzpicture}[scale=.5,baseline=.3cm]
		\node (1) at (0,.1) {$\scriptstyle 1$};
		\node (2) at (0,1.2) {$\scriptstyle 3$};
		\node (3) at (1,.5) {$\scriptstyle 2$};
		\blue{\draw [->] (.15,.9) .. controls (.4,.5) .. (.85,.5) ; 
		\draw [->] (.15,.9) .. controls (.25,.5) .. (.1,.1) ; }
		\red{\draw [->] (.85,.5) .. controls (.55,.8) .. (.15,.9) ; }
		\end{tikzpicture}
$$ 
which is not acyclic.
This orientation has the cycle $2\to 3\to 2$ and we detect that $x_2=x_3$ in $C$. 
We then set $B=\{2,3\}$ and redefine the family.
After going through one more iteration the algorithm will again give us the acyclic orientation $\begin{tikzpicture}[scale=.5,baseline=.3cm]
		\node (1) at (0,.1) {$\scriptstyle 1$};
		\node (3) at (.5,1) {$\scriptstyle 23$};
		\blue{\draw [->] (.5,.7) --(.1,.2) ; }
		\end{tikzpicture} $.
\end{example}

%\begin{remark} The polytope is a generalized permutahedron~\cite{Pos09}.
%In particular Postnikov gave a very nice formula for the volume of such polytope that can be interpreted as sum over generalized spanning hyperforest.\john{Is this remark still relevant? Should it be here? We do discuss a volume formula in Section~\ref{sec:PS}.}
%\end{remark}

We are now ready to connect this back with the antipode formula. If we look again at Theorem~\ref{thm:bb16} we notice that
the antipode formula is a sum over orientations.

\begin{corollary} \label{cor:geom_antipode}
For a hypergraph $G\in{\bf HG}[n]$, the coefficient of a flat $F\in Flats(G)$ in $S(G)$ is $a(G/F)$.
We have that $(-1)^na(G/F)$ is the Euler characteristic of the union of the faces of $P_G$ indexed by the acyclic orientations of $G/F$.
\end{corollary}

\begin{proof}
  This is a direct application of Theorem~\ref{THM:HypegraphicPolytope}. For a fixed flat $F\in{\bf HG}[n]$, Theorem~\ref{thm:bb16}
  tells us that
  $$ a(G/F)=\sum_{{\mathcal O} \in{\mathfrak O}(G/F)} (-1)^{|\Psi({\mathcal O})|}\,.$$
  This is up to a sign the Euler characteristic of the union of the faces of $P_G$ indexed by acyclic orientations of $G/F$.
  \end{proof}
  
   \begin{remark} For $G\in{\bf HG}[n]$ and $F=\emptyset$, we have $G/F=G$. Then the coefficient of $F$ in $S(G)$ is $(-1)^na(G)$ which is the Euler characteristic of a polytopal complex.
  This follows from the fact that if $A\vdash[n]$ is such that $\Omega(A)\in{\mathfrak O}(G)$, then for any refinement $B\le A$, we have $\Omega(B)\in{\mathfrak O}(G)$.
For any $F \in Flats(G)$, the coefficient of $F$ in $S(G)$ is the coefficient of $\emptyset$ in $S(G/F)$.
So, this coefficient of the antipode can be thought of in terms of the Euler characteristic of a polytopal complex where the polytopal complex may live in a smaller dimensional ambient space. 
Now, the full  antipode of the hypergraph $G\in{\bf HG}[n]$ can be thought of as a refinement of the Euler characteristic of $P_G.$
The Euler characteristic of $P_G$ is simply
\[\chi(P_G) = \sum_{f \subseteq P_G} (-1)^{\dim f} = 1\]
where the sum is over all faces of $P_G.$
The antipode formula is
\[S(G) = (-1)^n \sum_{f \subseteq P_G} (-1)^{\dim f} G_f \]
where the sum again runs over all faces of $P_G$ and $G_f$ denotes the flat of $G$ corresponding to the face $f$.
\end{remark}

\begin{example}\label{ex:SPG1}
For $G=
\begin{tikzpicture}[scale=.7,baseline=.2cm]
	\node (1) at (0,.1) {$\scriptstyle 1$};
	\node (2) at (0,1) {$\scriptstyle 3$};
	\node (3) at (1,.5) {$\scriptstyle 2$};
	\draw [fill=blue!40] (.1,.1) .. controls (.25,.5) .. (.15,.9) .. controls (.4,.5) .. (.85,.5) .. controls (.5,.4) .. (.1,.1) ; 
	\draw [color=red!60,thick] (.9,.7).. controls (0.6,1) ..(0.2,1);
\end{tikzpicture} 
$ as in Example~\ref{ex:PG1}, the flats of $G$ are $G, \{\{2,3\}\}, \emptyset$. The coefficient of each flat $F$ in $S(G)$ is given by the Euler characteristic of the faces of $P_G=\blue{{\bf \Delta}_{123}}+ \red{{\bf \Delta}_{23}}$ indexed by acyclic orientations of $G/F$:
$$\begin{array}{cccc}
&\begin{tikzpicture}[scale=1,baseline=.5cm]
	\draw [fill=gray!10,dotted ] (0,0) -- (.6,.75) -- (1.8,.75) -- (2.4,0) --(0,0) ; 
	\draw [fill=green!80!black] (.4,.2) -- (.7,.575) -- (1.7,.575) -- (2,.2) --(.4,.2) ; 
	\node at (1.2,.375) {$\scriptscriptstyle 123$};
\end{tikzpicture}  \quad &
 \begin{tikzpicture}[scale=1,baseline=.5cm]
	\draw [fill=gray!10,dotted ] (0,0) -- (.6,.75) -- (1.8,.75) -- (2.4,0) --(0,0) ; 
	\draw [thick,color=green!60!black] (0.1,0) --(2.3,0); 
	\draw [thick,color=green!60!black] (.7,.75) -- (1.7,.75); 
		\node at (1.2,-.4) {
	     $\begin{tikzpicture}[scale=.5,baseline=.5cm]
		\node (1) at (0,.1) {$\scriptstyle 1$};
		\node (3) at (.5,1) {$\scriptstyle 23$};
		\blue{\draw [->] (.5,.7) --(.1,.2) ; }
		\end{tikzpicture} $};
	\node at (1.2,1.2) {
	     $\begin{tikzpicture}[scale=.5,baseline=.5cm]
		\node (1) at (0,.1) {$\scriptstyle 1$};
		\node (3) at (.5,1) {$\scriptstyle 23$};
		\blue{\draw [<-] (.5,.7) --(.1,.2) ; }
		\end{tikzpicture} $};
\end{tikzpicture} \quad &
\begin{tikzpicture}[scale=1,baseline=.5cm]
	\draw [fill=gray!10,dotted ] (0,0) -- (.6,.75) -- (1.8,.75) -- (2.4,0) --(0,0) ; 
	\node at (0,0) {\textcolor{green!60!black}{$\scriptstyle \bullet$}};
	\node at (.6,.75) {\textcolor{green!60!black}{$\scriptstyle \bullet$}};
	\node at (1.8,.75) {\textcolor{green!60!black}{$\scriptstyle \bullet$}};
	\node at (2.4,0) {\textcolor{green!60!black}{$\scriptstyle \bullet$}};	
	\node (1) at (0.5,1.0) {
	     $\begin{tikzpicture}[scale=.5,baseline=.5cm]
		\node (1) at (0,.1) {$\scriptstyle 1$};
		\node (2) at (0,1.2) {$\scriptstyle 3$};
		\node (3) at (1,.5) {$\scriptstyle 2$};
		\blue{\draw [->] (.1,.1) .. controls (.25,.5) .. (.15,.9) ; 
		\draw [->] (.1,.1) .. controls (.5,.4) .. (.85,.5) ; }
		\red{\draw [->] (.85,.5) .. controls (.55,.8) .. (.25,.9) ; }
		\end{tikzpicture} $};
	\node  at (2.3,1.1) {
	     $\begin{tikzpicture}[scale=.5,baseline=.5cm]
		\node (1) at (0,.1) {$\scriptstyle 1$};
		\node (2) at (0,1.2) {$\scriptstyle 3$};
		\node (3) at (1,.5) {$\scriptstyle 2$};
		\blue{\draw [->] (.1,.1) .. controls (.25,.5) .. (.15,.9) ; 
		\draw [->] (.1,.1) .. controls (.5,.4) .. (.85,.5) ; }
		\red{\draw [<-] (.85,.6) .. controls (.55,.8) .. (.15,.9) ; }
		\end{tikzpicture} $};
	\node (2) at (-.2,-.2) {
	     $\begin{tikzpicture}[scale=.5,baseline=.5cm]
		\node (1) at (0,.1) {$\scriptstyle 1$};
		\node (2) at (0,1.2) {$\scriptstyle 3$};
		\node (3) at (1,.5) {$\scriptstyle 2$};
		\blue{\draw [->] (.85,.5) .. controls (.4,.5) .. (.15,.9) ; 
		\draw [->] (.85,.5) .. controls (.5,.4) .. (.1,.1) ; }
		\red{\draw [->] (.85,.5) .. controls (.55,.8) .. (.25,.9) ; }
		\end{tikzpicture} $};
	\node (3) at (2.8,-.3) {
	     $\begin{tikzpicture}[scale=.5,baseline=.5cm]
		\node (1) at (0,.1) {$\scriptstyle 1$};
		\node (2) at (0,1.2) {$\scriptstyle 3$};
		\node (3) at (1,.5) {$\scriptstyle 2$};
		\blue{\draw [->] (.15,.9) .. controls (.4,.5) .. (.85,.5) ; 
		\draw [->] (.15,.9) .. controls (.25,.5) .. (.1,.1) ; }
		\red{\draw [<-] (.85,.6) .. controls (.55,.8) .. (.15,.9) ; }
		\end{tikzpicture} $};
	\node at (2.6,.5) {
	     $\begin{tikzpicture}[scale=.5,baseline=.5cm]
		\node (1) at (0,.5) {$\scriptstyle 13$};
		\node (3) at (1,.5) {$\scriptstyle 2$};
		\blue{\draw [->] (.2,.5) -- (.8,.5) ; }
		\red{\draw [->] (.2,.5) .. controls (.55,.8) .. (.8,.6) ; }
		\end{tikzpicture} $};
	\node at (-.1,.6) {
	     $\begin{tikzpicture}[scale=.5,baseline=.5cm]
		\node (1) at (.5,.1) {$\scriptstyle 12$};
		\node (3) at (.1,1) {$\scriptstyle 3$};
		\blue{\draw [->] (.5,.1) -- (.2,.7) ; }
		\red{\draw [->] (.5,.1) .. controls (.5,.6) .. (.3,.7) ; }
		\end{tikzpicture} $};
	\draw [thick,color=green!60!black] (0,0) -- (.6,.75); 
	\draw [thick,color=green!60!black] (1.8,.75) -- (2.4,0); 
\end{tikzpicture}\\
S(G)=&\textcolor{green!60!black}{-1} \cdot \Big[
\begin{tikzpicture}[scale=.7,baseline=.2cm]
	\node (1) at (0,.1) {$\scriptstyle 1$};
	\node (2) at (0,1) {$\scriptstyle 3$};
	\node (3) at (1,.5) {$\scriptstyle 2$};
	\draw [fill=blue!40] (.1,.1) .. controls (.25,.5) .. (.15,.9) .. controls (.4,.5) .. (.85,.5) .. controls (.5,.4) .. (.1,.1) ; 
	\draw [color=red!80,thick] (.9,.7).. controls (0.6,1) ..(0.2,1);
\end{tikzpicture} \Big] 
& \textcolor{green!60!black}{+\quad2} \cdot \Big[\begin{tikzpicture}[scale=.7,baseline=.2cm]
	\node (1) at (0,.1) {$\scriptstyle 1$};
	\node (2) at (0,1) {$\scriptstyle 3$};
	\node (3) at (1,.5) {$\scriptstyle 2$};
	\draw [color=red!80,thick] (.9,.7).. controls (0.6,1) ..(0.2,1);
\end{tikzpicture} \Big]  &  \textcolor{green!60!black}{-\quad2} \cdot \Big[ \begin{tikzpicture}[scale=.7,baseline=.2cm]
	\node (1) at (0,.1) {$\scriptstyle 1$};
	\node (2) at (0,1) {$\scriptstyle 3$};
	\node (3) at (1,.5) {$\scriptstyle 2$};
\end{tikzpicture} \Big] \qquad \\
\end{array}
$$
\end{example}

\begin{example}\label{ex:SPG2} For the hypergraph
$G'=
\begin{tikzpicture}[scale=.7,baseline=.2cm]
	\node (1) at (0,.1) {$\scriptstyle 1$};
	\node (2) at (2.5,.5) {$\scriptstyle 4$};
	\node (3) at (0,1) {$\scriptstyle 2$};
	\node (4) at (1,.5) {$\scriptstyle 3$};
	\draw [fill=blue!40] (.1,.1) .. controls (.25,.5) .. (.15,.9) .. controls (.4,.5) .. (.85,.5) .. controls (.5,.4) .. (.1,.1) ; 
	\draw [thick,color=red!80] (4).. controls (1.8,.6) ..(2);
\end{tikzpicture} 
$ in Example~\ref{ex:PG2}, the flats are $G',\{\{3,4\}\},\{\{1,2,3\}\}$ and $\emptyset$. Thus
we  have that $S(G)$ is given by
$$\begin{array}{cccc}
\begin{tikzpicture}[scale=1,baseline=.5cm]
	\draw [fill=gray!10,dotted ] (0,0) -- (.6,.75)-- (1.6,1.25) -- (2.2,.5) -- (1.2,0) --(0,0) ; 
	\draw [dotted] (0,0)--(1,.5);
	\draw [dotted] (1.6,1.25)--(1,.5)--(2.2,.5);
	\draw [color=green!60!black,fill=green!60!black ] (.32,.125) -- (.74,.65)-- (1.16,.1250) --(.32,.125) ; 
	\draw [color=green!60!black,fill=green!60!black ]  (.74,.65)--(1.54,1.05)-- (1.96,.5250)-- (1.16,.1250)--(.74,.65) ; 
	\draw [densely dotted,color=green!40!black]  (1.16,.1250)--(.74,.65) ; 
	\draw [dotted]  (1.5,1.05)--(1.15,.525)--(1.92,.5250);
	\draw [dotted] (.6,.75)--(1.2,0);
\end{tikzpicture} \quad &
\begin{tikzpicture}[scale=1,baseline=.5cm]
	\draw [fill=gray!10,dotted ] (0,0) -- (.6,.75)-- (1.6,1.25) -- (2.2,.5) -- (1.2,0) --(0,0) ; 
	\draw [dotted] (0,0)--(1,.5);
	\draw [dotted] (1.6,1.25)--(1,.5)--(2.2,.5);
	\draw [color=green!60!black,fill=green!60!black] (.7,.8)-- (1.5,1.2) -- (2.1,.45) -- (1.3,0.05) --(.7,.8) ; 
	\draw [color=green!65!black,fill=green!65!black] (.1,0.05) -- (.7,.8) -- (1.3,0.05)--(.1,0.05) ; 
	\draw [thick,color=green!40!black] (.7,.8)-- (1.5,1.2); 
	\draw [thick,color=green!40!black] (2.1,.45) -- (1.3,0.05); 
	\draw [thick,color=green!40!black] (.1,0.05) -- (.92,0.46); 
	\draw [dotted] (.6,.75)--(1.2,0);
	\node (1) at (0,-.1) {
	     $\begin{tikzpicture}[scale=.3,baseline=.5cm]
		\blue{\draw [->] (.1,.1) -- (.15,.9) ; 
		\draw [->] (.1,.1) -- (.85,.5) ; }
		\end{tikzpicture} $};	
	\node (2) at (0.9,1.25) {
	     $\begin{tikzpicture}[scale=.3,baseline=.5cm]
		\blue{\draw [->]  (.15,.9) -- (.1,.1) ; 
		\draw [->] (.15,.9) -- (.85,.5) ; }
		\end{tikzpicture} $};	
	\node (12) at (0.1,.7) {
	     $\begin{tikzpicture}[scale=.3,baseline=.5cm]
		\blue{\draw [->] (.15,.5) -- (.85,.5) ; }
		\end{tikzpicture} $};	
	\node (3) at (1.8,.15) {
	     $\begin{tikzpicture}[scale=.3,baseline=.5cm]
		\blue{\draw [->]  (.85,.5) -- (.1,.1) ; 
		\draw [->] (.85,.5) --  (.15,.9); }
		\end{tikzpicture} $};	
	\node (13) at (.6,-.1) {
	     $\begin{tikzpicture}[scale=.3,baseline=.5cm]
		\blue{\draw [<-] (.5,.9) -- (.85,.5) ; }
		\end{tikzpicture} $};	
	\node (23) at (2.,1.2) {
	     $\begin{tikzpicture}[scale=.3,baseline=.5cm]
		\blue{\draw [<-] (.5,.1) -- (.85,.5) ; }
		\end{tikzpicture} $};	
\end{tikzpicture} \quad &
\begin{tikzpicture}[scale=1,baseline=.5cm]
	\draw [fill=gray!10,dotted ] (0,0) -- (.6,.75)-- (1.6,1.25) -- (2.2,.5) -- (1.2,0) --(0,0) ; 
	\draw [dotted] (0,0)--(1,.5);
	\draw [dotted] (1.6,1.25)--(1,.5)--(2.2,.5);
	\draw [dotted ](.6,.75)--(1.2,0);
	\draw [color=gray!10,fill=green!60!black] (.18,.075) -- (.6,.6)-- (1.02,.0750) --(.18,.075) ; 
	\draw [color=gray!10,fill=green!60!black ] (1.18,.575) -- (1.6,1.1)-- (2.02,.5750) --(1.18,.575) ; 
	\node (13) at (.5,0) {
	     $\begin{tikzpicture}[scale=.3,baseline=.5cm]
		\red{\draw [->] (.85,.5) -- (1.5,.5) ; }
		\end{tikzpicture} $};	
	\node at (2,1.3) {
	     $\begin{tikzpicture}[scale=.3,baseline=.5cm]
		\red{\draw [<-] (.85,.5) -- (1.5,.5) ; }
		\end{tikzpicture} $};	
\end{tikzpicture} \quad&
\begin{tikzpicture}[scale=1,baseline=.5cm]
	\draw [fill=gray!10,dotted ] (0,0) -- (.6,.75)-- (1.6,1.25) -- (2.2,.5) -- (1.2,0) --(0,0) ; 
	\draw [dotted] (0,0)--(1,.5);
	\draw [thick,color=green!60!black] (1.6,1.25)--(1,.5)--(2.2,.5);
	\draw [thick,color=green!60!black] (0,0) --(.6,.75)--(1.2,0)--(0,0);
	\draw [thick,color=green!60!black] (1.6,1.25) -- (2.2,.5);
	\node at (0,0) {\textcolor{green!60!black}{$\scriptstyle \bullet$}};
	\node at (.6,.75) {\textcolor{green!60!black}{$\scriptstyle \bullet$}};
	\node at (1.6,1.25) {\textcolor{green!60!black}{$\scriptstyle \bullet$}};
	\node at (2.2,.5) {\textcolor{green!60!black}{$\scriptstyle \bullet$}};	
	\node at (1.2,0) {\textcolor{green!60!black}{$\scriptstyle \bullet$}};	
	\node at (1,.5) {\textcolor{green!60!black}{$\scriptstyle \bullet$}};
	\node (1) at (-0.2,-.1) {
	     $\begin{tikzpicture}[scale=.3,baseline=.5cm]
		\blue{\draw [->] (.1,.1) -- (.15,.9) ; 
		\draw [->] (.1,.1) -- (.85,.5) ; }
		\red{\draw [->] (.85,.5) -- (1.5,.5) ; }
		\end{tikzpicture} $};	
	\node (2) at (0.5,1.1) {
	     $\begin{tikzpicture}[scale=.3,baseline=.5cm]
		\blue{\draw [->]  (.15,.9) -- (.1,.1) ; 
		\draw [->] (.15,.9) -- (.85,.5) ; }
		\red{\draw [->] (.85,.5) -- (1.5,.5) ; }
		\end{tikzpicture} $};	
	\node (12) at (0,.6) {
	     $\begin{tikzpicture}[scale=.3,baseline=.5cm]
		\blue{\draw [->] (.15,.5) -- (.85,.5) ; }
		\red{\draw [->] (.85,.5) -- (1.5,.5) ; }
		\end{tikzpicture} $};	
	\node (3) at (1.5,-.1) {
	     $\begin{tikzpicture}[scale=.3,baseline=.5cm]
		\blue{\draw [->]  (.85,.5) -- (.1,.1) ; 
		\draw [->] (.85,.5) --  (.15,.9); }
		\red{\draw [->] (.85,.5) -- (1.5,.5) ; }
		\end{tikzpicture} $};	
	\node (13) at (.5,-.1) {
	     $\begin{tikzpicture}[scale=.3,baseline=.5cm]
		\blue{\draw [<-] (.5,.9) -- (.85,.5) ; }
		\red{\draw [->] (.85,.5) -- (1.5,.5) ; }
		\end{tikzpicture} $};	
	\node (23) at (1.,.35) {
	     $\begin{tikzpicture}[scale=.3,baseline=.5cm]
		\blue{\draw [<-] (.5,.1) -- (.85,.5) ; }
		\red{\draw [->] (.85,.5) -- (1.5,.5) ; }
		\end{tikzpicture} $};	
	\node  at (1.3,.8) {
	     $\begin{tikzpicture}[scale=.3,baseline=.5cm]
		\blue{\draw [->] (.1,.1) -- (.15,.9) ; 
		\draw [->] (.1,.1) -- (.85,.5) ; }
		\red{\draw [<-] (.85,.5) -- (1.5,.5) ; }
		\end{tikzpicture} $};	
	\node  at (1.7,1.6) {
	     $\begin{tikzpicture}[scale=.3,baseline=.5cm]
		\blue{\draw [->]  (.15,.9) -- (.1,.1) ; 
		\draw [->] (.15,.9) -- (.85,.5) ; }
		\red{\draw [<-] (.85,.5) -- (1.5,.5) ; }
		\end{tikzpicture} $};	
	\node  at (1.4,1.15) {
	     $\begin{tikzpicture}[scale=.3,baseline=.5cm]
		\blue{\draw [->] (.15,.5) -- (.85,.5) ; }
		\red{\draw [<-] (.85,.5) -- (1.5,.5) ; }
		\end{tikzpicture} $};	
	\node  at (2.6,.6) {
	     $\begin{tikzpicture}[scale=.3,baseline=.5cm]
		\blue{\draw [->]  (.85,.5) -- (.1,.1) ; 
		\draw [->] (.85,.5) --  (.15,.9); }
		\red{\draw [<-] (.85,.5) -- (1.5,.5) ; }
		\end{tikzpicture} $};	
	\node at (1.7,.6) {
	     $\begin{tikzpicture}[scale=.3,baseline=.5cm]
		\blue{\draw [<-] (.5,.9) -- (.85,.5) ; }
		\red{\draw [<-] (.85,.5) -- (1.5,.5) ; }
		\end{tikzpicture} $};	
	\node at (2.2,1.2) {
	     $\begin{tikzpicture}[scale=.3,baseline=.5cm]
		\blue{\draw [<-] (.5,.1) -- (.85,.5) ; }
		\red{\draw [<-] (.85,.5) -- (1.5,.5) ; }
		\end{tikzpicture} $};	
\end{tikzpicture} \quad\\
\textcolor{green!60!black}{-1}\cdot \Big[
\begin{tikzpicture}[scale=.5,baseline=.2cm]
	\node (1) at (0,.1) {$\scriptstyle 1$};
	\node (2) at (2.5,.5) {$\scriptstyle 4$};
	\node (3) at (0,1) {$\scriptstyle 2$};
	\node (4) at (1,.5) {$\scriptstyle 3$};
	\draw [fill=blue!40] (.1,.1) .. controls (.25,.5) .. (.15,.9) .. controls (.4,.5) .. (.85,.5) .. controls (.5,.4) .. (.1,.1) ; 
	\draw [thick,color=red!80] (4).. controls (1.8,.6) ..(2);
\end{tikzpicture}  \Big]
&\textcolor{green!60!black}{+\quad 0}\cdot \Big[
 \begin{tikzpicture}[scale=.5,baseline=.2cm]
	\node (1) at (0,.1) {$\scriptstyle 1$};
	\node (2) at (2.5,.5) {$\scriptstyle 4$};
	\node (3) at (0,1) {$\scriptstyle 2$};
	\node (4) at (1,.5) {$\scriptstyle 3$};
	\draw [thick,color=red!80] (4).. controls (1.8,.6) ..(2);
\end{tikzpicture}  \Big]
 & \textcolor{green!60!black}{+\quad 2} \cdot \Big[
 \begin{tikzpicture}[scale=.5,baseline=.2cm]
	\node (1) at (0,.1) {$\scriptstyle 1$};
	\node (2) at (2.5,.5) {$\scriptstyle 4$};
	\node (3) at (0,1) {$\scriptstyle 2$};
	\node (4) at (1,.5) {$\scriptstyle 3$};
	\draw [fill=blue!40] (.1,.1) .. controls (.25,.5) .. (.15,.9) .. controls (.4,.5) .. (.85,.5) .. controls (.5,.4) .. (.1,.1) ; 
\end{tikzpicture}  \Big]
 &  \textcolor{green!60!black}{+\quad 0} \cdot \Big[ \begin{tikzpicture}[scale=.5,baseline=.2cm]
	\node (1) at (0,.1) {$\scriptstyle 1$};
	\node (2) at (2.5,.5) {$\scriptstyle 4$};
	\node (3) at (0,1) {$\scriptstyle 2$};
	\node (4) at (1,.5) {$\scriptstyle 3$};
\end{tikzpicture} \Big]
   \\
\end{array}
$$
\end{example}

One nice application of Corollary~\ref{cor:geom_antipode} is to continue~\cite[Example 4.5]{BB16}. Let us recall the definitions we need.

\begin{definition}\label{def:forest}
Given a hypergraph $G$, we say that $a_0{\buildrel U_1\over \longrightarrow}a_1{\buildrel U_2\over \longrightarrow}\cdots
{\buildrel U_\ell\over \longrightarrow}a_\ell$ is a \emph{path} of $G$ if $a_{i-1}\ne a_i$ and $\{a_{i-1},a_{i}\}\subset U_i\in G$ for each $1\le i\le \ell$.
We say that a path is \emph{proper} if  all the hyperedges $U_i$ are distinct.
A proper cycle in $G$ is a proper path such that $a_0=a_\ell$.
A hypergraph is a \emph{hyperforest} if it does not contain proper cycles. 
\end{definition}

We remark that if $G$ is a hyperforest, then the flats of $G$ precisely all possible subsetes of hyperedges $F\subseteq G$. 
%since any two edges intersect in at most ONE vertex, and there is no cycle,
% when you contract an edge, you do not create any further contraction. so any subset is a flat (the flat it spans does not contain any new edges).
The hyperedges $\{U_1, U_2, \dots, U_m\}$ of any hyperforest $G$ can be ordered so that 
\[|(U_1 \cup U_2 \cup \cdots \cup U_i) \cap U_{i+1}| \leq 1\]
for each $i$~\cite[Lemma 7]{Taylor}.
In this case, we obtain
  $$P_G=\prod_{U\in G}{\bf \Delta}_U.$$
In fact since the acyclic orientations of $G/F$ correspond to the boundary of $\prod_{U\in G/F} {\bf \Delta}_U$, we get the following proposition.

\begin{proposition}\cite[Prop 4.6]{BB16} Let $G$ be a hyperforest, $F$ a flat of $G$ and $k=|G/F|$. Also let $\ell$ be the number of connected components of $G/F$. Then
   $$a(G/F) = \begin{cases} 
      (-1)^\ell (-2)^k& \text{if   $\forall U\in G/F$ we have $|U|$ is even,}
     \\
     0& \text{otherwise.}
   \end{cases}
     $$
\end{proposition}

%%%%%%%%%%%%%%%%%%%%%%%%%%%
%%%%%%%%%%%%%%%%%%%%%%%%%%%
%%%%%%%%%%%%%%%%%%%%%%%%%%%
\section{Simple Polytopes}\label{S:polytope}

In this section we will consider certain families of polytopes: nestohedra, generalized Pitman-Stanley polytopes, and hyper-permutahedra. We will
use the correspondence between acyclic orientations and faces of hypergraphic polytopes from Theorem~\ref{THM:HypegraphicPolytope} to show that these polytopes are simple. Although some of these results are known our context provides a new perspective to study them.
In particular, we demonstrate that one is able to obtain information about Minkowski sums of simplices by only considering orientations of the underlying hypergraph.

Recall that for a hypergraph $G$ the set of its acyclic orientations is denoted $\mathfrak{O}(G)$. For each $k\in\{0,1,\dots,n-1\}$ define the set
\[\mathfrak{O}_k(G) := \{ \O \in \mathfrak{O}(G) : |V(G / \O)| = |V(G)| - k\}.\]
Observe that if $G$ has $n$ vertices then
\[\mathfrak{O}(G) = \bigsqcup_{k=0}^{n-1} \mathfrak{O}_k(G).\]
The \emph{1-skeleton} of a polytope $P$ is the graph consisting of the 0-dimensional and 1-dimensional faces of $P$. We denote the 1-skeleton of $P$ by $P^{(1)}$.
If $P$ is a $d$-dimensional polytope, then $P$ is called \emph{simple} if and only if $P^{(1)}$ is a $d$-regular graph. That is, if and only if every vertex of $P^{(1)}$ is incident to exactly $d$ edges.
By Theorem~\ref{THM:HypegraphicPolytope} the vertex set of $P_G^{(1)}$ is in 1-1 correspondence with $\mathfrak{O}_0(G)$ and the edge set is in 1-1 correspondence with
\[\mathfrak{O}_1(G) \sqcup \bigsqcup_{\substack{e \in G \\ |e| = 2}} \mathfrak{O}_0(G/e).\]

We now discuss the relationship between our results and work of Postnikov, Reiner, and Williams on generalized permutahedron.
The normal fan of any generalized permutahedron is known to be refined by the \emph{braid arrangement fan}~\cite[Proposition 3.2]{PRW}.
In the language of Postnikov-Reiner-Williams each cone in the normal fan of a generalized permutahedron is encoded by a preposet (i.e. a reflexive and transitive binary relation) while the normal fan is encoded by a \emph{complete fan of preposets}~\cite[Section 3]{PRW}.
In this context, our Theorem~\ref{THM:HypegraphicPolytope} says that when a generalized permutahedron is a hypergraphic polytope, the complete fan of preposets encoding its normal fan can be understood in terms of acyclic orientations.
Postnikov-Reiner-Williams~\cite[Corollary 3.6]{PRW} determine which complete fans of preposets correspond to complete fans of simplicial cones (and hence to simple polytopes).
They observe that cones of codimension 1 contained in a given cone of a normal fan are in bijection with the covering relations of the preposet corresponding to the cone~\cite[Proposition 3.5]{PRW}.
We state an equivalent result, translated to our language, for hypergraphic polytopes.

If $D$ is a directed acyclic graph we can think of it as a poset on its vertices and covering relations given by its edges.
We will denote the Hasse diagram of this poset $\Hasse(D)$ (i.e. the transitive reduction of $D$).
Let $G$ be a hypergraph, $F \in Flats(G)$, and $\O \in \mathfrak{O}(G/F)$.
We will identify the faces of $P_G$ and acyclic orientations via Theorem~\ref{THM:HypegraphicPolytope}.
The faces of $P_G$ containing $\O$ as a face of codimension 1 are then in bijection with edges of $\Hasse ((G/F)/\O)$.
Furthermore, if we contract a given edge $e$ of $\Hasse((G/F)/\O)$ there is a (necessarily unique) pair $(F', \O')$ such that $F' \in Flats(G)$, $\O'$ is an acyclic orientation of $G/F'$, and $(G/F')/\O'$ is equal to $(G/F)/\O$ contracted by $e$.
The pair $(F', \O')$ can be obtained using Lemma~\ref{lem:Omega}.
We record this result now as a lemma for later use.

\begin{lemma}
For $F \in Flats(G)$ and $\O \in \mathfrak{O}(G/F)$, the faces of $P_G$ containing the face indexed by $\O$ as a face of codimension 1 are in bijection with edges of $\Hasse((G/F)/\O)$, and each orientation $\O' \in \mathcal{AO}$ corresponding to such a face can be obtained by contracting $(G/F)/\O$ by an edge of $\Hasse((G/F)/\O)$.
\label{lem:simple}
\end{lemma}

\begin{theorem}
Let $G$ be a hypergraph.
The polytope $P_G$ is a simple polytope if and only if for every $\O \in \mathfrak{O}_0(G)$ the Hasse diagram $\Hasse(G / \O)$ is a forest.
\label{thm:simple}
\end{theorem}
\begin{proof}
If $G$ has $n$ vertices, then the dimension of $P_G$ is $n - c$ where $c$ is the number of connected components of $G$.
Observe that $G / \O$ will have $c$ connected components for any acyclic orientation $\O$, and hence $\Hasse( G / \O)$ will also have $c$ connected components.
Now $P_G$ is a simple polytope if and only if each vertex of $P_G$ is incident to exactly $n-c$ edges of $P_G$.
By Lemma~\ref{lem:simple} we know that the edges of the polytope $P_G$ incident to the vertex corresponding to $\O \in \mathfrak{O}_0(G)$ are in bijective correspondence to the edges of $\Hasse(G / \O)$.
The theorem follows since $\Hasse(G / \O)$ has $n - c$ edges if and only if it is a forest.
\end{proof}

When $G$ is a simple graph the graphic zonotope $P_G$ is simple if and only if the \emph{biconnected components} of $G$ are cliques~\cite[Proposition 5.2]{PRW}.
This is equivalent to $G$ being the \emph{line graph} of a forest~\cite[Remark 5.3]{PRW}. Theorem~\ref{thm:simple} gives a characterization of when a hypergraphic polytope is simple, but it is not always easy to verify the conditions of the theorem.
Nonetheless we now illustrate Theorem~\ref{thm:simple} with the forthcoming examples.

We now define building sets and nestohedra following~\cite{Pos09}.
A \emph{building set} $\mathcal{B}$ on $[n]$ is a collection of nonempty subsets of $[n]$ satsifying the following two conditions
\begin{enumerate}
\item[(i)] if $I,J \in \mathcal{B}$ and $I \cap J \neq \emptyset$, then $I \cup J \in \mathcal{B}$,
\item[(ii)] and $\{i\} \in \mathcal{B}$ for all $i \in [n]$.
\end{enumerate}
Given a building set $\mathcal B$ define the \emph{nestohedron} $P_{\mathcal{B}}$ as the Minkowski sum
\[ P_{\mathcal{B}} = \sum_{I \in \mathcal{B}} {\bf \Delta}_I.\]
For such $\mathcal{B}$ we will consider the hypergraph $G_{\mathcal{B}}$ with vertex set $[n]$ and hyperedge set consisting of $I \in \mathcal{B}$ such that $|I| \geq 2$.
The hypergraphic polytope $P_{G_{\mathcal{B}}}$ and the nestohedron $P_{\mathcal{B}}$ only differ by translation.

\begin{proposition}
Any nestohedron is a simple polytope.
\label{prop:nest}
\end{proposition}
\begin{proof}
Let $\mathcal{B}$ be any building set and let $G = G_{\mathcal{B}}$.
We will show for any $\O \in \mathfrak{O}_0(G)$ that $\Hasse(G / \O)$ is a forest.
The corollary will then follow from Theorem~\ref{thm:simple}.
In order for $\Hasse(G / \O)$ to be a forest, we must not be able to find a cycle in the underlying undirected graph.
In fact, we will show if we have a directed path from $b$ to $d$ and a directed path from $c$ to $d$ in $G / \O$, then we must also have a directed path from $b$ to $c$ or from $c$ to $b$ in $G / \O$.
This shows that in $\Hasse(G / \O)$ any vertex has in-degree at most $1$.
It follows that the underlying undirected graph of $\Hasse(G / \O)$ cannot contain a cycle since any acyclic orientation of a cycle graph must contain at least one vertex of in-degree $2$.

Assume that the sequences $b=u_0, u_1, u_2 \dots, u_p = d$ and $c=v_0, v_1, v_2, \dots, v_q = d$ give directed paths in $\Hasse(G / \O)$.
This means we have sequences of hyperedges $I_1, I_2, \dots, I_p$ and $J_1, J_2, \dots, J_q$ such that 
\begin{itemize}
\item $u_{i-1}, u_i \in I_i$ where $u_{i-1}$ is the source in $\O$
\item $v_{i-1}, v_i \in J_i$ where $v_{i-1}$ is the source in $\O$.
\end{itemize}
Since $\mathcal{B}$ was a building set and $\O$ is an acyclic orientation it follows that there exists a hyperedge $I$ containing $b$ and $d$ where $b$ is the source in $\O$, and there also exists a hyperedge $J$ containing $c$ and $d$ where $c$ is the source $\O$.
However, again using the fact that $\B$ is a building set we must have the hyperedge $I \cup J$ containing $b, c$, and $d$.
For the orientation $\O$ to be acyclic it follows either $b$ or $c$ must be the source of $I \cup J$.
Thus we must have either $(b,c)$ or $(c,b)$ in $G / \O$ and the proof is complete.
\end{proof}

\begin{remark}
Nestohedra are known to be simple from work of Postnikov~\cite[Theorem 7.4]{Pos09} and Feichtner-Sturmfels~\cite[Theorem 3.14]{FS05}.
\end{remark}

\subsection{Pitman-Stanley Polytopes}\label{sec:PS}

For any $n$ and $A \subseteq [n]$ with $n \in A$ we define the $(n-1)$-dimensional \emph{generalized Pitman-Stanley polytope} as the Minkowski sum
\[PS_{n,A} = \sum_{a \in A} {\bf\Delta}_{\{1,2,\ldots,a\}}.\]
Notice that $PS_{n,[n]}$ coincides with the  \emph{Pitman-Stanley polytope} from~\cite{Pitman-Stanley}.
We also observe that $PS_{n,A}$ is a translate of a nestohedron.

The Pitman-Stanley polytope $PS_{n,[n]}$ is closely related with parking functions. A \emph{parking function} of length $n$ is a sequence of nonnegative integers $\mathbf{a} = (a_1, a_2, \dots, a_n)$ such that $b_i \leq i-1$ where $b_1 \leq b_2 \leq \cdots \leq b_n$ is the increasing rearrangement of $\mathbf{a}$.
For any set of nonnegative integers $B$ we define $\Park_{n, B}$ to be the collection of parking functions of length $n$ which are sequences of elements taken from $B$.
Given a finite set of positive integers $A$ with $n = \max A$, define $\bar{A} := \{n - a : a \in A\}$.

\begin{proposition}
Consider $A = \{a_1 < a_2 < \cdots < a_k\}$ with $1 \not\in A$ and $n = \max A$.
The polytope $PS_{n,A}$ is a simple polytope with $f$-vector entries
\[f_j = \sum_{\substack{(\alpha_1, \alpha_2, \dots, \alpha_k) \\ 0 \leq \alpha_i \leq a_i - a_{i-1} \\ \alpha_1 + \alpha_2 + \cdots + \alpha_k = j}}\prod_{i=1}^k \binom{a_i - a_{i-1} + 1}{\alpha_i + 1}\]
where $a_0 = 1$.
Moreover, the normalized volume of $PS_{n,A}$ is given by
\[ \Vol(PS_{n,A}) = |\Park_{n-1,\bar{A}}|.\]
\label{prop:PS}
\end{proposition}
\begin{proof}
Let $A = \{a_1 < a_2 < \cdots < a_k\}$ and $G = \{[a_1], [a_2], \cdots, [a_k]\}$. Notice that $PS_{n,A} =  P_G$ is a hypergraphic polytope.
The polytope $PS_{n,A}$ is simple by Proposition~\ref{prop:nest} since $G$ is the hypergraph of a building set.
The flats of $G$ are of the form $F_i = G\big|_{([a_i], \{a_{i}+1\}, \dots, \{n\})}$  for $0 \leq i \leq k$.
The $j$-faces of $PS_{n,A}$ correspond to acyclic orientations $\O$ of $G / F$ where $F$ is a flat and $G/F$ has $n - j$ vertices.

If $\O$ is an acyclic orientation of $G / F_{i^*}$ for some $0 \leq i^* \leq k$, then for $i > i^*$ we set $\overline{[a_i]} = \{a_{i^*}, a_{i^*+1}, \dots, a_i\}$ to represent the hyperedge $[a_i]$ after contraction.
We obtain a sequence of sets $(S_1, S_2, \dots, S_k)$ where $S_i \subseteq ([a_i] \setminus [a_{i-1}]) \sqcup \{*\}$ for each $1 \leq i \leq k$.
We start by letting $S_i = ([a_i] \setminus [a_{i-1}]) \cup \{*\}$ for $1 \leq i \leq i^*$.
For $i^* < i \leq k$, we get the set $S_i$ by the following rule:
\begin{itemize}
\item If the sources of $\overline{[a_i]}$ in $\O$ are disjoint from the sources of $\overline{[a_{i-1}]}$ in $\O$, then let $S_i$ be the sources of  $\overline{[a_i]}$ in $\O$.
\item Otherwise the sources of $\overline{[a_i]}$ in $\O$ are not disjoint from the sources $\overline{[a_{i-1}]}$ in $\O$, and in this case we let $S_i$ be $*$ along with the sources of $\overline{[a_i]}$ in $\O$ which are in $\overline{[a_i]} \setminus \overline{[a_{i-1}]}$.
\end{itemize}

Given a sequence of sets $(S_1, S_2, \dots, S_k)$ where $S_i \subseteq ([a_i] \setminus [a_{i-1}]) \cup \{*\}$, we construct an orientation $\O$ as follows.
We let $i^*$ be chosen so that $S_{i^*+1} \neq ([a_{i^*+1}] \setminus [a_{i^*}]) \cup \{*\}$ but $S_i = ([a_i] \setminus [a_{i-1}]) \cup \{*\}$ for all $i < i^*$.
In this case we construct an orientation of $G / F_{i^*}$.
If $S_i \subseteq [a_i] \setminus [a_{i-1}]$, then we let the sources of $[a_i]$ be the the elements of $S_i$.
Otherwise if $* \in S_i$, then we let the sources of $[a_i]$ be the sources of $[a_{i-1}]$ along with the elements of $S_i \setminus \{*\}$.

The two processes are inverse to each other.
We have used the fact that if $e \subset f$ are hyperedges, then in any acyclic orientation the sources of $f$ must either contain all sources of $e$ or must be disjoint.
It is clear that there are
\[\sum_{\substack{(\alpha_1, \alpha_2, \dots, \alpha_k) \\ 0 \leq \alpha_i \leq a_i - a_{i-1} \\ \alpha_1 + \alpha_2 + \cdots + \alpha_k = j}}\prod_{i=1}^k \binom{a_i - a_{i-1} + 1}{\alpha_i + 1}\]
such sequences of sets.
The result on the $f$-vector follows.

It remains to compute the volume of $P_G$.
Since $G$ is a connected hypergraph on $n$ vertices,
it follows from~\cite[Corollary 9.4]{Pos09} that the normalized volume of the hypergraphic polytope $P_G$ is equal to the number of sequences
$(e_1,e_2, \dots, e_{n-1})$ of hyperedges of $G$ such that $|e_{i_1} \cup e_{i_2} \cup \cdots \cup e_{i_k}| \geq k+1$ for any distinct $i_i, i_2, \cdots, i_k$.
We will exhibit a bijection between the set of such sequences and $\Park_{n-1,\bar{A}}$.
We claim that the map 
\[ (e_1,e_2, \dots, e_{n-1}) \mapsto (n-|e_1|,n-|e_2|, \dots, n-|e_{n-1}|)\]
gives this desired bijection between the sequences of hyperedges contributing to the volume of $P_G$ and $\Park_{n-1,\bar{A}}$.
The inverse map is
\[(a_1,a_2, \cdots, a_{n-1}) \mapsto (e_1, e_2, \dots,e_{n-1})\]
where $e_i$ is the unique hyperedge in $G$ with $|e_i| = n-a_i$.
For a sequence of hyperedges $(e_1,e_2, \dots, e_{n-1})$ it is clear each $n-|e_i| \in \bar{A}$ if each $e_i \in G$.
Let $f_1 \subseteq f_2 \subseteq \cdots \subseteq f_{n-1}$ be the increasing rearrangement of the sequence of hyperedges $(e_1, e_2, \dots, e_{n-1})$.
In order for this sequence to contribute to the volume we must have $|f_i| \geq i + 1$.
Since $|f_{n-i}| \geq n-i+1$ if and only if $n-|f_{n-i}| \leq i-1$ the result follows.
\end{proof}

Now let us apply Proposition~\ref{prop:PS} when $n = mk+1$ and $A = \{k+1,2k+1, \dots, n\}$.
In this case the $f$-vector entries of $PS_{n,A}$ are given by
\[
     f_j = \sum_{\substack{(\alpha_1, \alpha_2, \dots, \alpha_m) \\ 0 \leq \alpha_i \leq k \\ \alpha_1 + \alpha_2 + \cdots + \alpha_m = j}}\prod_{i=1}^m \binom{k+1}{\alpha_i + 1}.
\]
By letting $b_i = | \{ \ell : a_{\ell} = i\}|$ we obtain
\begin{equation}
     f_j = \sum_{\substack{b_0, b_1, \cdots, b_k \geq 0\\b_0 + b_1 + \cdots + b_k = m\\b_1 + 2b_2 + \cdots +kb_k = j}} \binom{m}{b_0, b_1, \dots, b_k} \prod_{i = 0}^k \binom{k+1}{i+1}^{b_i}.
     \label{eq:f}
\end{equation}
 From either of these expressions we can observe that the number of vertices of such a polytope is $f_0 = (k+1)^m$.
 Also, the number of facets of this $mk$-dimensional polytope is $f_{mk-1} = m(k+1)$. 
 
 \begin{example}[k=1] If $k = 1$ we have $A = \{2,3,\dots,n\}$ and thus the polytope $PS_{n,A}$ conincides with the Pitman-Stanley polytope. Proposition ~\ref{prop:PS} tells us that the $f$-vector entries of $PS_{n,A}$ are given by
 \[   f_j = \sum_{\substack{b_0, b_1 \geq 0\\b_0 + b_1 = n-1\\b_1=j}} \binom{n-1}{b_0,b_1} \binom{2}{1}^{b_0} \binom{2}{2}^{b_1} = 2^{n-1-j}\binom{n-1}{j}.\]
 which agrees with the $f$-vector of an $(n-1)$-dimensional hypercube.
 The Pitman-Stanley polytope $PS_{n,[n]}$ is known to be combinatorially equivalent to an $(n-1)$-dimensional hypercube~\cite[Theorem 19]{Pitman-Stanley}.
 \end{example}

 \begin{example}[k=2] The case when $k=2$ and $n = 2m + 1$ with $A = \{3, 5, \dots, n\}$ gives us that the $f$-vector entries are
 \begin{align*}
     f_j &= \sum_{\substack{b_0, b_1, b_2 \geq 0\\b_0 + b_1 + b_2 = m\\b_1 + 2b_2 = j}}\binom{m}{b_0,b_1,b_2} \binom{3}{1}^{b_0}\binom{3}{2}^{b_1}\binom{3}{3}^{b_2}
 \end{align*}
 When $j = 2j'$ is even we obtain
 \begin{align*}
     f_j &= \sum_{r=0}^{j'} \binom{m}{m-j'-r,2r,j'-r}3^{m-j'+r}\\
     &= 3^{m-j'}\sum_{r=0}^{j'} \binom{m}{m-j'-r,2r,j'-r}3^r.
 \end{align*}
If $j = 2j'+1$ is odd we obtain
 \begin{align*}
     f_j &= \sum_{r=0}^{j'} \binom{m}{m-j'-r-1,2r+1,j'-r}3^{m-j'+r}\\
     &= 3^{m-j'}\sum_{r=0}^{j'} \binom{m}{m-j'-r-1,2r+1,j'-r}3^r.
 \end{align*}
We see that $3$ divides $f_j$ for $0 \leq j \leq 2m$.
This generalizes to sets $A = \{(p-1) + 1, 2(p-1) + 1, \dots n\}$ where $n = m(p-1) + 1$ and $p$ is prime.
\end{example}

\begin{proposition}
Let $n = m(p-1) + 1$ for some prime $p$ and let 
\[A = \{(p-1) + 1, 2(p-1) + 1, \dots n\}.\]
Then the polytope $PS_{n,A}$ is $m(p-1)$-dimensional and its $f$-vector entries satisfy
\[f_j \equiv 0 \pmod{p} \]
for $0 \leq j < m(p-1)$.
\label{prop:modp}
\end{proposition}
\begin{proof}
From Equation~(\ref{eq:f}) we see that
\[ f_j = \sum_{\substack{b_0, b_1, \cdots, b_{p-1} \geq 0\\b_0 + b_1 + \cdots + b_{p-1} = m\\ b_1 + 2b_2 + \cdots +(p-1)b_{p-1} = j}} \binom{m}{b_0, b_1, \dots b_{p-1}} \prod_{i = 0}^{p-1} \binom{p}{i+1}^{b_i}.\]
Since $0 \leq j < m(p-1)$ each term in the sum must have $b_i \neq 0$ for some $0 \leq i < p-1$, and hence this term will be divisble by $p$ since it will have a factor of $\binom{p}{i+1}$ which is divisible by $p$.
\end{proof}

\subsection{Hyper-permutahedra}\label{ss:hyper-perm}
In~\cite{Ag16} Agnarsson studies a class of generalized permutahedra which are hypergraphic polytopes.
Polytopes in this class are called \emph{hyper-permutahedra} and defined by $\Pi_{n-1}(k-1) := P_G$ for $G = \binom{[n]}{k}$.
We always assume $k\ge 2$.
Hyper-permutahedra are known to be simple polytopes~\cite[Proposition 2.4]{Ag16}.
We now give another proof that $\Pi_{n-1}(k-1)$ is simple using acyclic orientations in hypergraphs.

\begin{proposition}
The hyper-permutahedron $\Pi_{n-1}(k-1)$ is a simple polytope.
\label{prop:hyperperm}
\end{proposition}

\begin{proof}
Let $G = \binom{[n]}{k}$ and consider $\O \in \mathfrak{O}_0(G)$.
We claim that there is a unique set composition $A=(\{a_1\},\{a_2\},\ldots, \{a_{n-k+1}\}, B)\models[n]$
such that $\Omega(A)=\O$ and the $\Hasse(G / \O)$ is a tree with edges
\[\{(a_j,a_{j+1}) : 1 \leq j < n-k+1\} \cup \{(a_{n-k+1},b) : b\in B\}.\]
We proceed by induction on $n-k$. As a base case, first assume $n = k$.
In this situation $G = \{[n]\}$ is a single hyperedge and for any $\O \in \mathfrak{O}_0$ we have $\O= \big\{(\{a\}, [n] \setminus \{a\}) \big\}$ for a unique $a\in [n]$. The
 $\Hasse(G / \O)$ has edges
\[\big\{(a,b) : b\in [n] \setminus \{a\} \big\}.\]
Hence the set composition $A=(\{a\},[n] \setminus \{a\})$ determines $\Hasse(G / \O)$ and  $\Omega(A)=\O$.

Assume now that $n-k > 0$ and take $\O \in \mathfrak{O}_0(G)$. Recall that by definition $\O \in \mathfrak{O}_0(G)$ is such that $V(G/\O)=[n]$.
We have that $G/\O$ is an acyclic directed graph, and thus $G/ \O$ has a source $a\in[n]$. 
Suppose $a'\ne a$ is another source of $G/ \O$. Let $e\in \binom{[n]}{k}=G$ be such that $a,a'\in e$.
The orientation  of $e$ in $\O$ must be such that both $a$ and $a'$ are the head of $e$, a contradiction.
We must then have a unique source and we let $a_1=a$. Note that for any $b\in [n]\setminus\{a\}$, there is a hyperedge containing both $a$ and $b$. Thus 
$(a, b)$ is an edge in $G / \O$.
Next consider $G' = G \setminus \{a_1\}$ which is isomorphic to $\binom{[n-1]}{k}$.
The orientation $\O \in \mathfrak{O}_0(G)$ corresponds to an orientation $\O' \in \mathfrak{O}_0(G')$ by forgetting any hyperedges containing $a_1$.
By induction hypothesis $\Hasse(G' / \O')$ is obtained from $A'=(\{a_2\},\ldots,\{a_{n-k+1}\},B)$ such that $\Omega(A')=\O'$.
Since $a_1$ is smaller than all element of $\Hasse(G' / \O')$ with unique minimal element $a_2$, it follows that $\Hasse(G / \O)$ is obtained from $A=(\{a_1\},\{a_2\},\ldots, \{a_{n-k+1}\}, B)\models[n]$
as required and by construction $\Omega(A)=\O$.
We have $\Hasse(G / \O)$ is a tree and so $\Pi_{n-1}(k-1)$ is a simple polytope by Theorem~\ref{thm:simple}.
Since the sequence of sources $(a_1,a_2,\ldots,a_{n-k+1})$ are unique at each stage, we have that the set composition $A$ is unique.
\end{proof}

In the proof above, we saw that  $\O \in \mathfrak{O}_0(G)$ determines a unique set composition $A=(\{a_1\},\ldots, \{a_{n-k+1}\}, B)\models[n]$.
The converse is also true: given any $A=(\{a_1\},\ldots, \{a_{n-k+1}\}, B)\models[n]$, the orientation $\Omega(A)\in \mathfrak{O}_0(G)$.
It then follows that $|\mathfrak{O}_0(G)| = (n-k+1)!\binom{n}{k-1}=\frac{n!}{(k-1)!}$ and therefore $\Pi_{n-1}(k-1)$ has $\frac{n!}{(k-1)!}$ vertices.
We now illustrate this with an example.

\begin{example}
Let $G = \binom{[5]}{3}$.
We consider the orientation $\O \in \mathfrak{O}_0(G)$ consisting of:
\[\begin{array}{c}
(\{2\}, \{1,3\}) \quad (\{2\}, \{1,4\}) \quad (\{2\}, \{1,5\}) \quad (\{4\}, \{1,3\}) \quad (\{1\}, \{3,5\})\\
(\{4\}, \{1,5\}) \quad (\{2\}, \{3,4\}) \quad (\{2\}, \{3,5\}) \quad (\{2\}, \{4,5\}) \quad (\{4\}, \{3,5\})
\end{array}\]
This orientation corresponds to the set composition $A =( \{2\}, \{4\}, \{1\}, \{3,5\})$.
The Hasse diagram $\Hasse(G / \O)$ is shown in Figure~\ref{fig:Hasse}.
\label{ex:hyperperm}
\end{example}

\begin{figure}
\centering
\begin{tikzpicture}
\node (2) at (0,0) {$2$};
\node (4) at (0,-1) {$4$};
\node (1) at (0,-2) {$1$};
\node (3) at (-1,-3) {$3$};
\node (5) at (1,-3) {$5$};
\draw[->] (2) to (4);
\draw[->] (4) to (1);
\draw[->] (1) to (3);
\draw[->] (1) to (5);
\end{tikzpicture}
\label{fig:Hasse}
\caption{The Hasse diagram $\Hasse(G / \O)$ from Example~\ref{ex:hyperperm}.}
\end{figure}
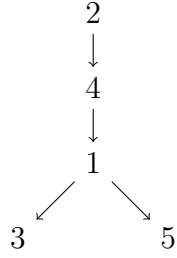

An \emph{ordered pseudo-partition (OPP)}~\cite[Definition 4.9]{Ag16} of $[n]$ is a sequence of sets $(A_0,A_1, \dots, A_p,B)$ where:
\begin{itemize}
    \item $[n] = A_0 \sqcup A_1 \sqcup \cdots \sqcup A_p  \sqcup B$.
    \item $A_0, A_1, \dots, A_p \neq \emptyset$.
\end{itemize}
Here $B$ is allowed to be empty.

\begin{remark}
Our ordering of the parts slightly differs from~\cite[Definition 4.9]{Ag16}.
We place $B$ as the last part rather than the first part.
This placement of $B$ fits more natural with notions we have developed around acyclic orientations.
\end{remark}
Let $\PSC_{n,k,j}$ denote the collection of ordered pseudo-partitions $(A_0, \dots, A_p,B)$ of $[n]$ with:
\begin{itemize}
    \item $0 \leq |B| \leq k-1$.
    \item $k \leq |B| + |A_p| \leq n$.
    \item $n-j = |B| + p + 1$.
\end{itemize}

\begin{proposition}[{\cite[Theorem 4.10]{Ag16}}]
The hyper-permutahedron $\Pi_{n-1}(k-1)$ has $f$-vector entries
\[f_j = |\PSC_{n,k,j}|.\]
\end{proposition}
\begin{proof}
Let $G = \binom{[n]}{k}$ and take $\O \in \mathfrak{O}_0(G)$.
From Proposition~\ref{prop:hyperperm} and its proof we know that $\Hasse(G / \O)$ is obtained from a unique set composition $A=(\{a_1\},\ldots, \{a_{n-k+1}\}, B)$.
This gives us $(A_0, \dots, A_p,B)\in \PSC_{n,k,0}$ where $A_i=\{a_{i+1}\}$. Conversely, given $(A_0, \dots, A_p,B)\in \PSC_{n,k,0}$, we must have that
 $$|B|=n-p-1\le k-1 \qquad\implies\qquad n-k\le p.$$
 The only possibility is if $|B|=k-1$, $p=n-k$ and $|A_i|=1$ and this gives us a unique $\O \in \mathfrak{O}_0(G)$. Hence $f_0=|\PSC_{n,k,0}|$.
 
 For the 1-faces, we know that they are obtained by contracting a single edge of $\Hasse(G / \O)$ for all  $\O \in \mathfrak{O}_0(G)$.
 For any $\O \in \mathfrak{O}_0(G)$, we have its set composition $A=(\{a_1\},\ldots, \{a_{n-k+1}\}, B)$.
 There are two types of edges in $\Hasse(G / \O)$. If we contract an edge $(a_i,a_{i+1})$ then we obtain a unique OPP 
  $$A'=(\{a_1\},\ldots,\{a_i,a_{i+1}\},\ldots \{a_{n-k+1}\}, B)\in \PSC_{n,k,1}.$$
  If we contract an edge $(a_{n-k+1},b)$ then we obtain
  $$A'=(\{a_1\},\ldots,\{a_{n-k+1},b\}, B\setminus\{b\})\in \PSC_{n,k,1}.$$
  Conversely, given $A'\in \PSC_{n,k,1}$ there are exactly two possible Hasse diagrams that can contract to it giving us a 1-face. Hence $f_1=|\PSC_{n,k,1}|$.
We can continue this iteration to show that $f_j=|\PSC_{n,k,j}|$ for $j>1$.
The process also clarifies why we want the possibility of $B$ to be empty in the definition of OPPs.

Indeed one sees that the data any OPP $(A_0, A_1, \dots, A_p, B)$ is equivalent to a poset on $\{A_i : 0 \leq i \leq p\} \sqcup B$ (i.e. a preposet on $[n]$) where 
\begin{enumerate}
\item[(i)] $A_i > A_{i'}$ for each $i < i'$, 
\item[(ii)] $A_i > b$ for each $i$ and $b \in B$, 
\item[(iii)] and the $b$ and $b'$ are incomparable for any distinct $b, b' \in B$.
\end{enumerate}
Moreover, if $(A_0, A_1, \dots, A_p, B) \in \PSC_{n,k,j}$ after contracting any edge in the Hasse diagram of the equivalent preposet just described we will obtain a preposet equivalent to an element of $\PSC_{n,k,j+1}$.
We can see that any element of $\PSC_{n,k,j}$ can be obtain by $j$ contractions of a poset equivalent to an element of $\PSC_{n,k,0}$.
For a given $(A_0, A_1, \dots, A_p, B) \in \PSC_{n,k,j}$ first choose any linear order $A_i$ for $1 \leq i \leq p$.
Next declare that $a \in A_i$ is greater that $a' \in A_i'$ for any $i < i'$.
Lastly add that relations that $a \in A_i$ for any $i$ is greater than each $b \in B$.
There are many ways on doing this depending on the number of flags of faces in the hyper-permutahdron.
We can also see that any $j$ contractions in the Hasse diagram of the poset corresponding to an element of $\PSC_{n,k,0}$ will result is a preposet equivalent to the data of an element of $\PSC_{n,k,j}$
\end{proof}

\bibliographystyle{alpha}
\bibliography{AntiChromQuasi}
\end{document}